\documentclass[a4paper,leqno,10pt]{amsart}

\usepackage{epsf,graphicx,epsfig}
\usepackage{amscd}
\usepackage{amsmath,latexsym,amssymb,amsthm}
\usepackage[nospace,noadjust]{cite}
\usepackage{textcomp}
\usepackage{setspace,cite}
\usepackage{lscape,fancyhdr,fancybox}
\usepackage{stmaryrd}
\usepackage[all,cmtip]{xy}
\usepackage{tikz}
\usepackage{cancel}
\usetikzlibrary{shapes,arrows,decorations.markings}
\usepackage{graphicx}

\usepackage{color}
\usepackage{url}
\usepackage{enumerate}
\usepackage[mathscr]{euscript}

  \theoremstyle{plain}

\swapnumbers
    \newtheorem{thm}{Theorem}[section]
    \newtheorem{prop}[thm]{Proposition}
     
   \newtheorem{lemma}[thm]{Lemma}

    \newtheorem{subsec}[thm]{}
\theoremstyle{definition}
    \newtheorem{defn}[thm]{Definition}
        \newtheorem{remark}[thm]{Remark}
    \newtheorem{exam}[thm]{Example}

\theoremstyle{remark}

\title{}
\author{}
\date{}
\usepackage{amssymb}

\usepackage{hyperref}
\hypersetup{
	colorlinks,
	citecolor=blue,
	filecolor=black,
	linkcolor=blue,
	urlcolor=black
}

\begin{document}
\title{Involutive and oriented dendriform algebras}
\author{Apurba Das}
\address{Department of Mathematics and Statistics,
Indian Institute of Technology, Kanpur 208016, Uttar Pradesh, India.}
\email{apurbadas348@gmail.com}

\author{Ripan Saha}
\address{Department of Mathematics,
Raiganj University, Raiganj, 733134,
West Bengal, India.}
\email{ripanjumaths@gmail.com}

\begin{abstract}
Dendriform algebras are certain splitting of associative algebras and arise naturally from Rota-Baxter operators, shuffle algebras and planar binary trees. In this paper, we first consider involutive dendriform algebras, their cohomology and homotopy analogs. The cohomology of an involutive dendriform algebra splits the Hochschild cohomology of an involutive associative algebra. In the next, we introduce a more general notion of oriented dendriform algebras. We develop a cohomology theory for oriented dendriform algebras that closely related to extensions and governs the simultaneous deformations of dendriform structures and the orientation.
\end{abstract}
\maketitle

\tableofcontents

\section{Introduction}
In this paper, we study dendriform algebras equipped with an involution or more generally equipped with an orientation. Classical algebras such as associative algebras, $A_\infty$-algebras and $L_\infty$-algebras with involutions often appear in the standard constructions of algebras arising in geometric contexts, when the underlying geometric object has an involution \cite{cos,braun}. For instance, the de Rham cohomology of a manifold with an involution carries an involutive $A_\infty$-algebra structure \cite{loday-val}. However, such involutive algebras first appeared in mathematical physics in the context of an unoriented version of topological conformal field theories \cite{cos}. Motivated from the fact that cyclic involutive $A_\infty$-algebras are equivalent to open Klein topological conformal field theories, Braun \cite{braun} studied Hochschild cohomology of involutive associative algebras (more generally of involutive $A_\infty$-algebras). In \cite{fer-gian} the authors gave a simple interpretation of Braun's involutive Hochschild cohomology using involutive Bar complex which led them to also introduce Hochschild homology of involutive associative algebras. 

Recently, Koam and Pirashvili \cite{koam-pira} introduced a notion of oriented associative algebras (provided by an orientation group) and a new cohomology for an oriented associative algebra that governs the simultaneous deformation of the algebra and the orientation. An oriented associative algebra with the orientation group $\{ \pm 1 \}$ is an involutive associative algebra. 

As mentioned at the beginning, in this paper, we shall study involutive and oriented dendriform algebras. Dendriform algebras first appeared in the work of Loday in his study of the periodicity phenomenons in the algebraic $K$-theory \cite{loday}. He also defines cohomology for dendriform algebras with trivial coefficients. The operadic approach of the cohomology with arbitrary coefficients was treated in \cite{loday-val}. An explicit description of the cohomology was given in \cite{das}. Dendriform algebras pay very much attention in last two decades due to its connection with Rota-Baxter operators \cite{aguiar,fard-guo}, shuffle algebras \cite{loday,ronco} and planar binary trees \cite{loday}. Subsequently, Loday and his collaborators defined few other algebras (e.g. associative dialgebras, associative trialgebras, dendriform trialgebras, quadri-algebras, ennea-algebras etc.) in connections with combinatorial objects \cite{loday,loday-ronco,aguiar-loday,leroux}. See the recent article \cite{das-loday} for overview of these algebras. These algebras are called Loday-type algebras.

The contents of the present paper section-wise can be described as follows.
In Section \ref{section-prelim}, we recall some preliminaries about dendriform algebras and their cohomology.
In Section \ref{section-inv-dend}, we first define a cohomology for involutive dendriform algebras. 
This cohomology can be thought of as a splitting of Hochschild cohomology of involutive associative algebras (Theorem \ref{inv-split-ass}). We also show that the ordinary dendriform cohomology of an involutive dendriform algebra splits as a sum of this involutive dendriform cohomology and a skew factor (Theorem \ref{thm-decom}). 
Next, we consider involutive strongly homotopy dendriform algebras. They are splitting of involutive $A_\infty$-algebras introduced in \cite{braun}. Rota-Baxter operators on involutive $A_\infty$-algebras give rise to involutive strongly homotopy dendriform algebras. We end Section \ref{section-inv-dend} by defining involutions in other Loday-type algebras. Examples of involutive dendriform algebras will be given in Section \ref{section-ori-dend} where we consider a more general notion of oriented dendriform algebras.

In Section \ref{section-cohomo-ori-dend}, we define a cohomology theory for oriented dendriform algebras. In Proposition \ref{ass-dend-mor}, we relate this cohomology with the cohomology of oriented associative algebras introduced in \cite{koam-pira}. Like classical cohomology theories, the cohomology of oriented dendriform algebras is closely related to extensions of oriented dendriform algebras (Theorem \ref{ext-thm}).

Finally, in Section \ref{section-defor-ori-dend}, we generalize the classical deformation theory of Gerstenhaber \cite{gers} to oriented dendriform algebras. More precisely, we consider deformations of oriented dendriform algebras by deforming both the dendriform structure and the given orientation. We will see that such deformations are closely related to the cohomology of oriented dendriform algebras introduced in the previous section (Theorems \ref{formal-def} and \ref{inf-def}).

All vector spaces, linear maps and tensor products are over a field $\mathbb{K}$ of characteristics zero.

\section{Preliminaries}\label{section-prelim}
In this section, we recall dendriform algebras, their representations and explicit cohomology theory \cite{loday}, \cite{loday-val}, \cite{das}. We also recall group cohomology of group $G$ with coefficients in a $G$-module  (see \cite{weibel} for instance).

\subsection{Dendriform algebras}
\begin{defn}
A dendriform algebra consists of a vector space $D$ together with two linear maps $\prec, \succ : D \otimes D \rightarrow D$ satisfying the following three identities
\begin{align}
( a \prec b ) \prec c  =~&  a \prec ( b \prec c + b \succ c), \label{dend-iden-1}\\
( a \succ b ) \prec c =~& a \succ ( b  \prec c ),\\
( a \prec b + a \succ b) \succ c =~& a \succ ( b \succ c), ~\text{for all } a, b, c \in D.  \label{dend-iden-3}
\end{align}
\end{defn}

A dendriform algebra as above may be denoted by a triple $(D, \prec, \succ)$. It follows from (\ref{dend-iden-1})-(\ref{dend-iden-3}) that the sum operation $a \star b = a \prec b + a \succ b$ turns out to be associative. Thus, a dendriform algebra can be thought of as a splitting of an associative algebra.

Dendriform algebras appear from Rota-Baxter algebras, shuffle algebras and arithmetre of planar binary trees \cite{aguiar,fard-guo,ronco}. The operadic cohomology of dendriform algebras was given in \cite{loday-val}.
In \cite{das} one of the present author describes the explicit cohomology theory for dendriform algebras with coefficients in a representation. We recall this cohomology here. For $n \geq 1$, let $C_n$ be the set of first $n$ natural numbers. We denote the elements of $C_n$ as $C_n = \{ [1], [2], \ldots, [n] \}. $ For
$m , n \geq 1$ and $1 \leq i \leq m,$
there are maps $R_0 (m ; 1, \ldots, n, \ldots, 1) : C_{m+n-1} \rightarrow C_m$ and $R_i (m ; 1, \ldots, n, \ldots, 1) : C_{m+n-1} \rightarrow \mathbb{K}[C_n]$ defined by
\begin{align}\label{r-0} R_0 (m; \overbrace{1, \ldots, \underbrace{ n}_{i\text{-th ~ place}},  \ldots, 1}^m) ([r]) ~=~
\begin{cases} [r] ~~~ &\text{ if } ~~ r \leq i-1 \\ [i] ~~~ &\text{ if } i \leq r \leq i +n -1 \\
[r -n + 1] ~~~ &\text{ if } i +n \leq r \leq m+n -1, \end{cases}
\end{align}
\begin{align}\label{r-i} R_i (m; \overbrace{1, \ldots,  \underbrace{ n}_{i\text{-th ~ place}},  \ldots, 1}^m) ([r]) ~=~
\begin{cases} [1] + [2] + \cdots + [n] ~~~ &\text{ if } ~~ r \leq i-1 \\ [r - (i-1)] ~~~ &\text{ if } i \leq r \leq i +n -1 \\
[1]+ [2] + \cdots + [n] ~~~ &\text{ if } i +n \leq r \leq m+n -1. \end{cases}
\end{align}

\begin{defn}
Let $D= (D, \prec, \succ)$ be a dendriform algebra. A representation of $D$ consists of a vector space $M$ together with linear maps (called action maps)
\begin{align*}
\prec : D \otimes M \rightarrow M \qquad \succ : D \otimes M \rightarrow M \qquad
\prec : M \otimes D \rightarrow M \qquad \succ :  M \otimes D \rightarrow M
\end{align*}
satisfying the $9$ identities where each pair of $3$ identities correspond to the identities (\ref{dend-iden-1})-(\ref{dend-iden-3}) with
exactly one of $a, b, c$ is replaced by an element of $M.$
\end{defn}

It follows that any dendriform algebra is a representation of itself with the obvious action maps.
Before we recall the cohomology with coefficients in a representation, we need some more notations. Define elements $\pi_D \in \mathrm{Hom}(\mathbb{K}[C_2] \otimes D {\otimes} D , D)$, $\theta_1 \in \mathrm{Hom}(\mathbb{K}[C_2] \otimes D \otimes M, M)$ and $\theta_2 \in \mathrm{Hom}(\mathbb{K}[C_2] \otimes M \otimes D, M)$ by
\begin{align*}
\begin{cases} \pi_D ([1]; a, b ) = a \prec b  \\ 
\pi_D ([2]; a, b) = a \succ b, \end{cases} \qquad
\begin{cases} \theta_1 ([1]; a, m ) = a \prec m  \\ 
\theta_1 ([2]; a, m) = a \succ m, \end{cases} \qquad
\begin{cases} \theta_2 ([1]; m, a ) = m \prec a  \\ 
\theta_2 ([2]; m,a) = m \succ a, \end{cases}
\end{align*}
for $a, b \in D$ and $m \in M$. Then the linear maps $\prec, \succ : D {\otimes } D \rightarrow D$ defines a dendriform algebra structure on $D$ if and only if $\pi_D$ satisfies 
\begin{align*}
\pi_D ( R_0 (2;2,1)[r]; \pi_D (R_1 (2;2,1)[r]; a, b), c) = \pi_D ( R_0 (2;1,2)[r]; a, \pi_D (R_2 (2;1,2)[r]; b, c),
\end{align*}
for all $a, b, c \in D$ and $[r] \in C_3$. Moreover, $M$ defines a representation of the dendriform algebra $D$ if and only if for $a, b \in D$, $m \in M$ and $[r] \in C_3$,
\begin{align*}
\theta_1 ( R_0 (2;2,1)[r]; \pi_D (R_1 (2;2,1)[r]; a, b), m) =~& \theta_1 ( R_0 (2;1,2)[r]; a, \theta_1 (R_2 (2;1,2)[r]; b, m),
\\
\theta_2 ( R_0 (2;2,1)[r]; \theta_1 (R_1 (2;2,1)[r]; a, m), b) =~& \theta_1 ( R_0 (2;1,2)[r]; a, \theta_2 (R_2 (2;1,2)[r]; m, b),
\\
\theta_2 ( R_0 (2;2,1)[r]; \theta_2 (R_1 (2;2,1)[r]; m,a), b) =~& \theta_2 ( R_0 (2;1,2)[r]; m, \pi_D (R_2 (2;1,2)[r]; a,b).
\end{align*}
For more details, see \cite{das}. Let $(D, \prec, \succ)$ be a dendriform algebra and $M$ be a representation. For each $n \geq 1$, we define
\begin{align*}
C^n_{\mathrm{dend}}(D, M) := \mathrm{Hom} (\mathbb{K}[C_n] \otimes D^{\otimes n}, M)
\end{align*}
and a map $\delta_{\mathrm{dend}} : C^n_{\mathrm{dend}}(D, M) \rightarrow C^{n+1}_{\mathrm{dend}}(D, M)$ by
\begin{align*}
&(\delta_{\mathrm{dend}} f) ([r]; a_1 , \ldots, a_{n+1}) \\
&=  \theta_1 \big( R_0 (2;1,n) [r]; ~a_1, f (R_2 (2;1,n)[r]; a_2, \ldots, a_{n+1})   \big) \\
&+ \sum_{i=1}^n  (-1)^i~ f \big(  R_0 (n; 1, \ldots, 2, \ldots, 1)[r]; a_1, \ldots, a_{i-1}, \pi_D (R_i (1, \ldots, 2, \ldots, 1)[r]; a_i, a_{i+1}), a_{i+2}, \ldots, a_{n+1}   \big) \\
&+ (-1)^{n+1} ~\theta_2 \big( R_0 (2; n, 1) [r];~ f (R_1 (2;n,1)[r]; a_1, \ldots, a_n), a_{n+1}   \big),
\end{align*}
for $[r] \in C_{n+1}$ and $a_1, \ldots, a_{n+1} \in D$. It is shown in \cite{das} that $(\delta_{\mathrm{dend}})^2 = 0$. The cohomology of the cochain complex $( \{ C^n_{\mathrm{dend}} (D, M) \}_{n \geq 1}, \delta_{\mathrm{dend}} )$ is called the cohomology of the dendriform algebra $D$ with coefficients in $M$. We call this cohomology as ordinary cohomology to distinguish it from the cohomology of involutive dendriform algebras in the next section.

\subsection{Group cohomology}
In the main course of the paper, we will also need the cohomology of a group with coefficients in a module \cite{weibel}. Let $G$ be a group and $M$ a $G$-module. We define $C^n_{\mathrm{grp}}(G, M) := \mathrm{Maps}( G^{\times n}, M)$ and a map $\delta_{\mathrm{grp}} : C^n_{\mathrm{grp}}(G, M) \rightarrow C^{n+1}_{\mathrm{grp}}(G, M)$ by
\begin{align*}
(\delta_{\mathrm{grp}} f) ( g_1, \ldots, g_{n+1}) =  g_1 f( g_2 , \ldots, g_{n+1}) + \sum_{i=1}^n (-1)^i~ f ( g_1, \ldots, g_i g_{i+1}, \ldots, g_{n+1}) + (-1)^{n+1} f ( g_1, \ldots, g_n).
\end{align*}
Then it follows that $(\delta_{\mathrm{grp}} )^2 = 0$. The corresponding cohomology groups are denoted by $H^n_{\mathrm{grp}}(G, M)$, for $n \geq 0$ and they are called the cohomology of $G$ with coefficients in the module $M$.

\begin{remark}\label{rem-grp-coho}
A cochain complex $(M_\bullet, \delta)$ is said to be $G$-equivariant if each $M_i$ is a $G$-module and the coboundary map $\delta$ is $G$-equivariant. In this case, one can define a bicomplex computing the cohomology of $G$ with coefficients in the complex $(M_\bullet, \delta)$. The $i$-th column of the bicomplex is the cohomology complex of the group $G$ with coefficients in $M_i$ and the horizontal maps are induced by the $G$-equivariant map $\delta$.
\end{remark}

\section{Involutive dendriform algebras and their cohomology}\label{section-inv-dend}

In this section, we introduce involutive dendriform algebras and define suitable cohomology for involutive dendriform algebras. This cohomology is a splitting of the Hochschild cohomology of involutive associative algebras introduced by Braun \cite{braun}. The ordinary dendriform cohomology of an involutive dendriform algebra splits as a sum of this involutive dendriform cohomology and a skew factor. We also introduce involutive $\mathrm{Dend}_\infty$-algebras and their relation with Rota-Baxter operator on involutive $A_\infty$-algebras. We end this section by defining involutions in other Loday-type algebras.

\begin{defn}
Let $D = (D, \prec, \succ)$ be a dendriform algebra. An involution on $D$ is a linear map $ *:  D \rightarrow D, ~ a \mapsto a^*$ satisfying $a^{**} = a$ and
$ (a \prec b)^* = b^*  \succ a^*,$ for all $a, b \in D$.
\end{defn}
The two conditions in the above definition implies that $(a \succ b)^* = b^* \prec a^*$ as 
\begin{align*}
( a \succ b )^* = (a^{**} \succ b^{**})^* = (b^* \prec a^*)^{**} = b^* \prec a^*.
\end{align*}
A dendriform algebra together with an involution is called an involutive dendriform algebra. Examples of involutive dendriform algebras will be provided in the next section when we define a more general notion of oriented dendriform algebras.

An involutive associative algebra is an associative algebra $A$ together with a linear map $*  : A \rightarrow A, ~ a \mapsto a^*$ satisfying $a^{**} = a$ and $(ab)^* = b^* a^*$, for all $a, b \in A$. It follows that if $(D, \prec, \succ)$ is an involutive dendriform algebra with involution $*$, then $(D, \star = \prec + \succ)$ is an involutive associative algebra with same involution $*$.

Let $(D, \prec, \succ)$ be an involutive dendriform algebra with involution $*$. A representation of it is given by a representation $M$ of the dendriform algebra $D$ together with an involution $M \rightarrow M, ~m \mapsto m^*$ (i.e. $m^{**} = m$) satisfying $(a \prec m)^* = m^* \succ a^*$ and $( a \succ m)^* = m^* \prec a^*$. 

These conditions also implies that $(m \prec a)^* = a^* \succ m^*$ and $(m \succ a)^* = a^* \prec m^*$ holds. It is easy to see that any involutive dendriform algebra is a representation of itself.

\subsection{Cohomology}\label{subsec-inv-cohomo}

Let $(D, \prec, \succ)$ be an involutive dendriform algebra and $M$ be a representation. First consider the ordinary dendriform algebra cochain complex $(\{ C^n_{\mathrm{dend}} (D, M) \}_{n \geq 1}, \delta_{\mathrm{dend}} )$ as given in the previous section. For each $n \geq 1$, we define
\begin{align*}
iC^n_{\mathrm{dend}} (D, M) := \big\{ f \in C^n_{\mathrm{dend}} (D, M) |~ f ([r]; a_1, \ldots, a_n)^* = (-1)^{\frac{(n-1)(n-2)}{2} } f ([n-r+1]; a_n^*, \ldots, a_1^*),\\
\forall ~[r] \in C_n~ \text{ and } a_1, \ldots, a_n \in D \big\}.
\end{align*}

\begin{prop}\label{inv-sub-comp}
We have $\delta_{\mathrm{dend}} (iC^n_{\mathrm{dend}} (D, M)) \subset iC^{n+1}_{\mathrm{dend}} (D, M) $, for all $n \geq 1$.
\end{prop}

\begin{proof}
For $f \in iC^n_{\mathrm{dend}} (D, M) $ and $[r] \in C_{n+1}$, we have
\begin{align*}
&((\delta_{\mathrm{dend}} f) ([r]; a_1, \ldots, a_{n+1}))^* \\
&= \theta_1 \big( R_0 (2;1,n) [r]; ~a_1, f (R_2 (2;1,n)[r]; a_2, \ldots, a_{n+1})   \big)^* \\
& + \sum_{i=1}^n  (-1)^i~ f \big(  R_0 (n; 1, \ldots, 2, \ldots, 1)[r]; a_1, \ldots, a_{i-1}, \pi_D (R_i (n;1, \ldots, 2, \ldots, 1)[r]; a_i, a_{i+1}), a_{i+2}, \ldots, a_{n+1}   \big)^* \\
& + (-1)^{n+1} ~\theta_2 \big( R_0 (2; n, 1) [r];~ f (R_1 (2;n,1)[r]; a_1, \ldots, a_n), a_{n+1}   \big)^* \\
&= \theta_2 \big( R_0 (2;n,1) [n-r+2]; f (R_2 (2;1,n)[r]; a_2, \ldots, a_{n+1})^*, a_1^*   \big) \\
& + (-1)^{\frac{(n-1)(n-2)}{2}}~ \sum_{i=1}^n (-1)^i ~ f \big( R_0 (n; 1, \ldots, \underbrace{2}_{n-i+1}, \ldots, 1)[n-r+2]; a_{n+1}^*, \ldots, a_{i+2}^*,\\
& \qquad \qquad \qquad \qquad \qquad \qquad \pi_D (R_i (n;1, \ldots, \underbrace{2}_i, \ldots, 1)[r]; a_i, a_{i+1})^*, a_{i-1}^*, \ldots, a_1^* \big)  \\
& + (-1)^{n+1} \theta_1 (R_0 (2;1,n)[n-r+2]; a_{n+1}^*, f (R_1 (2;n,1)[r]; a_1, \ldots, a_n)^*) \\
&= (-1)^{\frac{(n-1)(n-2)}{2}}~ \theta_2 \big( R_0 (2;n,1) [n-r+2]; f (R_1 (2;n,1)[n-r+2]; a_{n+1}^*, \ldots, a_{2}^*), a_1^*   \big) \\
&+ (-1)^{\frac{(n-1)(n-2)}{2}} (-1)^{n+1} ~\sum_{i=1}^n (-1)^{n-i+1}~ f \big( R_0 (n; 1, \ldots, \underbrace{2}_{n-i+1}, \ldots, 1)[n-r+2]; a_{n+1}^*, \ldots, a_{i+2}^*, \\
& \qquad \qquad \qquad \qquad \qquad \qquad  \pi_D (R_{n-i+1} (n;1, \ldots, \underbrace{2}_{n-i+1}, \ldots, 1)[n-r+2];  a_{i+1}^*, a_i^*), a_{i-1}^*, \ldots, a_1^* \big)  \\
& + (-1)^{n+1}  (-1)^{\frac{(n-1)(n-2)}{2}}~ \theta_1 (R_0 (2;1,n)[n-r+2]; a_{n+1}^*, f (R_2 (2;1,n)[n-r+2]; a_n^*, \ldots, a_1^*)) \\
&= (-1)^{\frac{(n-1)(n-2)}{2} + (n+1)} ~(\delta_{\mathrm{dend}} f)([n-r+2]; a_{n+1}^*, \ldots, a_1^*) \\
&= (-1)^{\frac{n(n-1)}{2}}  ~ (\delta_{\mathrm{dend}} f)([n-r+2]; a_{n+1}^*, \ldots, a_1^*).
\end{align*}
This shows that $\delta_{\mathrm{dend}} f \in iC^{n+1}_{\mathrm{dend}} (D, M)$.
\end{proof}

It follows from the above proposition that $( \{ iC^n_{\mathrm{dend}} (D, M) \}_{n \geq 1}, \delta_{\mathrm{dend}} )$  is a subcomplex of the cochain complex $(\{ C^n_{\mathrm{dend}} (D, M) \}_{n \geq 1}, \delta_{\mathrm{dend}} )$. The cohomology of this subcomplex is called the cohomology of the involutive dendriform algebra $D$ with coefficients in $M$.

Note that the cohomology of involutive associative algebras was introduced by Braun \cite{braun} by viewing an involutive associative algebra as a square-zero derivation on the free involutive associative algebra (equivalently as a square-zero coderivation on the cofree involutive associative coalgebra). Later, Fern\'{a}ndez-Val\'{e}ncia and Giansiracusa \cite{fer-gian} gave a construction of the cohomology using involutive bar complex. Here we give a more explicit description of their cohomology.

Let $A$ be an involutive associative algebra and $M$ be an involutive $A$-bimodule. Then the collection of spaces $\{ iC^n_{\mathrm{Hoch}} (A, M) \}_{n \geq 0}$ where $iC^0_{\mathrm{Hoch}} = \{ m \in M| m^* = -m \}$ and for $n \geq 1$
\begin{align*}
iC^n_{\mathrm{Hoch}} (A, M) := \big\{ f : A^{\otimes n} \rightarrow M |~ f ( a_1, \ldots, a_n)^* = (-1)^{\frac{(n-1)(n-2)}{2} } f ( a_n^*, \ldots, a_1^*),~ \forall a_1, \ldots, a_n \in A \big\}
\end{align*}
is a subcomplex of the classical Hochschild cochain complex. The cohomology of this subcomplex is called the Hochschild cohomology of the involutive associative algebra $A$ with coefficients in the involutive $A$-bimodule $M$. We denote the corresponding cohomology groups by $iH^\bullet_{\mathrm{Hoch}}(A, M)$.

It has been shown in \cite[Theorem 2.9]{das} that ordinary dendriform algebra cohomology splits the Hochschild cohomology. This splitting passes onto respective involutive subcomplexes. Hence we obtain the following.

\begin{thm}\label{inv-split-ass}
Let $(D, \prec, \succ)$ be an involutive dendriform algebra and $M$ a representation of it. Consider the corresponding involutive associative algebra $(D, \star = \prec + \succ) $ and the involutive associative $D$-bimodule $M$. Then the collection of maps
\begin{align*}
S_n : iC^n_{\mathrm{dend}}(D, M) \rightarrow iC^n_{\mathrm{Hoch}} (D, M), ~ f \mapsto f_{[1]} + \cdots + f_{[n]},~ \text{ for } n \geq 1
\end{align*}
induces a morphism $iH^\bullet_{\mathrm{dend}}(D, M) \rightarrow iH^\bullet_{\mathrm{Hoch}} (D, M)$ between corresponding cohomologies.
\end{thm}

\subsection{Splitting theorem} Let $(D, \prec, \succ)$ be an involutive dendriform algebra and $M$ be a representation of it. For each $n \geq 1$, consider a linear map $T_n : C^n_{\mathrm{dend}}(D, M) \rightarrow C^n_{\mathrm{dend}} (D, M)$ defined by
\begin{align*}
(T_n f) ([r]; a_1, \ldots, a_n ) = (-1)^{\frac{(n-1)(n-2)}{2} } ~f( [n-r+1]; a_n^*, \ldots, a_1^*)^*,
\end{align*}
for all $[r] \in C_n$ and $a_1, \ldots, a_n \in D$. It is easy to see that $(T_n)^2 = \mathrm{id}$. Hence the map $T_n$ has eigenvalues $+1$ and $-1$. The eigenspace corresponding to the eigenvalue $+1$ is precisely given by $iC^n_{\mathrm{dend}}(D, M)$. We denote the eigenspace corresponding to the eigenvalue $-1$ by $i_{-}C^n_{\mathrm{dend}}(D, M)$. Then it is easy to see that
\begin{align}\label{split-iso}
C^n_{\mathrm{dend}}(D, M) \cong iC^n_{\mathrm{dend}}(D, M) \oplus i_{-}C^n_{\mathrm{dend}}(D, M), ~ f \mapsto \bigg( \frac{f + T_n f}{2}, \frac{f - T_nf}{2} \bigg), \text{ for } n \geq 1.
\end{align}
Moreover, a similar observation of Proposition \ref{inv-sub-comp} shows that $( \{ i_{-}C^n_{\mathrm{dend}}(D, M) \}_{n \geq 1}, \delta_{\mathrm{dend}} ) $ is a subcomplex of the ordinary dendriform algebra complex $( \{ C^n_{\mathrm{dend}}(D, M) \}_{n \geq 1}, \delta_{\mathrm{dend}} ) $. We denote the corresponding cohomology groups by $i_{-}H^\bullet_{\mathrm{dend}}(D, M)$. It is also easy to see that the isomorphisms (\ref{split-iso}) preserve the corresponding differentials on both sides. Hence we obtain the following.

\begin{thm}\label{thm-decom}
Let $D$ be an involutive dendriform algebra and $M$ be a representation of it. Then the ordinary dendriform algebra cohomology $H^\bullet_{\mathrm{dend}} (D, M)$ splits as $H^\bullet_{\mathrm{dend}}(D, M) \cong iH^\bullet_{\mathrm{dend}}(D, M) \oplus i_{-}H^\bullet_{\mathrm{dend}}(D, M)$.
\end{thm}

\begin{remark}
A similar decomposition theorem for the cohomology of involutive associative algebras (with coefficients in itself) was obtained by Braun \cite{braun}. His approach is based on viewing as involutive associative algebra as a square-zero derivation on the free involutive associative algebra (equivalently as a square-zero coderivation on the cofree associative coalgebra). This approach is only applicable to cohomology with self coefficients. On the other hand, our approach is more elementary and applicable to cohomology with arbitrary coefficients.
\end{remark}

\begin{remark}
It has been shown in \cite{das} that the ordinary dendriform cohomology of a dendriform algebra with coefficients in itself governs the deformation of the dendriform structure. For the general operadic treatment, see \cite{loday-val}. Similarly, the cohomology of an involutive dendriform algebra $D$ (with coefficients in itself) defined in this section governs the deformation of $D$ as an involutive dendriform algebra (not allowing to deform the involution). In the following sections, we consider a more general notion of oriented dendriform algebras and their cohomology that governs the simultaneous deformation of the dendriform structure and the orientation.
\end{remark}

\subsection{Involutive $\mathrm{Dend}_\infty$-algebras}\label{subsec-inv-dend-inf}
The notion of strongly homotopy dendriform algebras ($\mathrm{Dend}_\infty$-algebras in short) was defined in the book of Loday and Vallette \cite{loday-val}. In \cite{das}, one of the present author describes $\mathrm{Dend}_\infty$-algebras using the combinatorial functions $\{ R_0, R_i \}$ that are defined in Section \ref{section-prelim}. In this subsection, we study involutive $\mathrm{Dend}_\infty$-algebras.

\begin{defn}\cite{das} 
A $\mathrm{Dend}_\infty$-algebra consists of a graded vector space $\mathcal{A} = \oplus A_i$ together with a collection $\{ \mu_{k, [r]} : \mathcal{A}^{\otimes k} \rightarrow \mathcal{A} ~| ~1 \leq k < \infty, [r] \in C_k \}$ of multilinear maps with $\mathrm{deg}(\mu_{k, [r]}) =k-2$ satisfying the following identities: for each $n \geq 1$ and $[r] \in C_n$,
\begin{align*}
\sum_{i+j= n+1} \sum_{\lambda = 1}^j~(-1)^{\lambda (i+1) + i (|a_1|+ \cdots + |a_{\lambda -1}|) } ~\mu_{j, R_0(j;1, \ldots, i, \ldots, 1)[r]} \big( a_1, \ldots, a_{\lambda -1},
\mu_{i, R_\lambda (j;1, \ldots, i, \ldots, 1)[r]} (a_\lambda, \ldots, a_{\lambda + i -1} ),\\ a_{\lambda + i}, \ldots, a_n \big) = 0, 
\end{align*}
for $a_1, \ldots, a_n \in \mathcal{A}.$
\end{defn}

The coderivation interpretation of $\mathrm{Dend}_\infty$-algebras are also given in \cite{das}. Consider the graded vector space $V = s \mathcal{A}$, the suspension of $\mathcal{A}$. Then the multiplications $\{\mu_{k, [r]} \}_{1 \leq k < \infty, [r] \in C_k}$ induce multiplications $\{ \varrho_{k, [r]} : V^{\otimes k} \rightarrow V ~|~ 1 \leq k < \infty,~ [r] \in C_k\}$ on the graded vector space $V$ with $\mathrm{deg} (\varrho_{k, [r]}) = -1$. More precisely, $\varrho_{k, [r]} := (-1)^{\frac{k(k-1)}{2}} ~s \circ \mu_{k, [r]} \circ (s^{-1})^{ \otimes k }$. Finally, the multiplications $\{ \varrho_{k, [1]}, \ldots, \varrho_{k, [k]} \}$ induces a degree $-1$ coderivation $\widetilde{\varrho_k}$ on the free diassociative coalgebra $\mathrm{Diass}^c(V) = TV \otimes V \otimes TV$. It has been shown that $(\mathcal{A}, \{ \mu_{k, [r]} \})$ is a  $\mathrm{Dend}_\infty$-algebra is equivalent to that the degree $-1$ coderivation
\begin{align*}
D = \widetilde{\varrho_1} + \widetilde{\varrho_2} + \cdots \in \mathrm{Coder}^{-1}(\mathrm{Diass}^c(V)) \quad \mathrm{  satisfies } \quad D^2 = 0.
\end{align*}

\begin{defn}
An involutive $\mathrm{Dend}_\infty$-algebra is a $\mathrm{Dend}_\infty$-algebra $(\mathcal{A}, \{ \mu_{k, [r]} \})$ together with an involution (i.e. a degree $0$ linear map $* : \mathcal{A} \rightarrow \mathcal{A}, ~ a \mapsto a^*$ satisfying $a^{**} = a$) that satisfies
\begin{align}\label{inv-dend-inf-iden}
\mu_{k, [r]}(a_1, \ldots, a_k)^*  = (-1)^\theta (-1)^{\frac{(n-1)(n-2)}{2}} ~\mu_{k, [k-r+1]} (a_k^*, \ldots, a_1^*), ~ \text{ for } 1 \leq k < \infty ~ \text{ and } [r] \in C_k,
\end{align}
where $\theta = \sum_{i=1}^k |a_i| (\sum_{j=i+1}^k |a_j| )$ appears from the Koszul sign convension.
\end{defn}

An involutive $\mathrm{Dend}_\infty$-algebra whose underlying graded vector space $\mathcal{A}$ is concentrated in degree $0$ is nothing but an involutive dendriform algebra.

\begin{remark}
An involution on the graded vector space $\mathcal{A}$ gives rise to an involution on $V = s \mathcal{A}$ that induces an involution on the free diassociative coalgebra $\mathrm{Diass}^c(V)$. Then $(\mathcal{A}, \{ \mu_{k, [r]} \}, * )$ is an involutive $\mathrm{Dend}_\infty$-algebra if and only if $D = \widetilde{\varrho_1} + \widetilde{\varrho_2} + \cdots $ is a square-zero coderivation on $\mathrm{Diass}^c(V)$ that preserves the involution, i.e. $D(x^*) = D(x)^*$, for $x \in \mathrm{Diass}^c(V)$.
\end{remark}

Involutive $A_\infty$-algebras was introduced was introduced in \cite{braun} as an $A_\infty$-algebra $(\mathcal{A}, \{ \mu_k \})$ together with an involution that satisfies the identity similar to (\ref{inv-dend-inf-iden}), i.e.
\begin{align*}
\mu_{k}(a_1, \ldots, a_k)^*  = (-1)^\theta (-1)^{\frac{(n-1)(n-2)}{2}} ~\mu_{k} (a_k^*, \ldots, a_1^*), ~ \text{ for } 1 \leq k < \infty.
\end{align*}
Hence as a consequence of \cite[Theorem 4.8]{das}, we get the following.

\begin{prop}
Let $(\mathcal{A}, \{ \mu_{k, [r]} \}, *)$ be an involutive $\mathrm{Dend}_\infty$-algebra. Then $(\mathcal{A}, \{ \mu_{k} \}, *)$ is an involutive $A_\infty$-algebra, where $\mu_k :=  \mu_{k, [1]} + \cdots + \mu_{k, [k]}$, for $k \geq 1$.
\end{prop}

\begin{remark}\label{remark-inf-cohomo}
It is known that skeletal $\mathrm{Dend}_\infty$-algebras are closely related to the ordinary cohomology of dendriform algebras \cite{das}. In the same manner, one can easily show that skeletal involutive $\mathrm{Dend}_\infty$-algebras are related to the cohomology of involutive dendriform algebras introduced in Subsection \ref{subsec-inv-cohomo}. Since the proof is along to the line of above-mentioned reference, we do not repeat it here.
\end{remark}

Rota-Baxter operators on $A_\infty$-algebras was introduced in \cite{das} that induce $\mathrm{Dend}_\infty$-algebra structures. Here we extend it to the involutive case.

\begin{defn}
Let $(\mathcal{A}, \{ \mu_k \}, *)$ be an involutive $A_\infty$-algebra.  A degree $0$ linear map $R : \mathcal{A} \rightarrow \mathcal{A}$ is said to be a Rota-Baxter operator if $R$ satisfies $R(a^*) = R(a)^*$ and
\begin{align*}
R (\mu_k (a_{1}, \ldots, a_k)) =  R \bigg( \sum_{i=1}^k  \mu_k (R(a_1), \ldots, R(a_{i-1}), a_i, R(a_{i+1}), \ldots, R(a_k)) \bigg), ~ \text{ for } k \geq 1. 
\end{align*}
\end{defn}

\begin{prop}
Let $R$ be a Rota-Baxter operator on an involutive $A_\infty$-algebra $(\mathcal{A}, \{ \mu_k \}, *)$. Then $(\mathcal{A}, \{ \mu_{k, [r]} \}, *)$ is an involutive $\mathrm{Dend}_\infty$-algebra, where
\begin{align*}
\mu_{k, [r]} (a_1, \ldots, a_k ) :=  \mu_k ( R(a_1), \ldots, R(a_{r-1}), a_r, R(a_{r+1}), \ldots, R(a_k)), \text{ for } k \geq 1 \text{ and } [r] \in C_k.
\end{align*}
\end{prop}

\begin{proof}
The pair $(\mathcal{A}, \{ \mu_{k, [r]} \})$ is a $\mathrm{Dend}_\infty$-algebra is proved in \cite[Theorem 5.2]{das}. Moreover, we have
\begin{align*}
\mu_{k, [r]} (a_1, \ldots, a_k)^* =~& \mu_k ( R(a_1), \ldots, R(a_{r-1}), a_r, R(a_{r+1}), \ldots, R(a_k))^* \\
=~& (-1)^\theta (-1)^{\frac{(n-1)(n-2)}{2}}~ \mu_k (R(a_k^*), \ldots, R(a_{r+1}^*), a_r^*, R(a_{r-1}^*), \ldots, R(a_1^*)) \\
=~& (-1)^\theta (-1)^{\frac{(n-1)(n-2)}{2}}~ \mu_{k,[k-r+1]} (a_k^*, \ldots, a_1^*),~\text{ for } [r] \in C_k.
\end{align*}
Hence we have verified the identity (\ref{inv-dend-inf-iden}).
\end{proof}

\subsection{Involutions in other Loday-type algebras}
There are other algebras (than dendriform algebras) introduced by Loday and his collaborators in their study of algebras arising in combinatorics. Associative dialgebras, associative trialgebras, dendriform trialgebras, quadri-algebras, ennea-algebras are examples of such algebras \cite{loday,loday-ronco,aguiar-loday,leroux}. In \cite{das-loday} it has been shown that all these algebras can be described by multiplications in suitable operads. These algebras are called Loday-algebras. More precisely, for any fixed type of Loday-algebras, there is a collection of non-empty sets $\{ U_n \}_{n \geq 1}$ and a collection of `nice' functions 
\begin{align*}
R_0 (m ; \overbrace{1, \ldots,1, \underbrace{n}_{i\text{-th}} ,1, \ldots, 1 }^m) :~& U_{m+n-1} \rightarrow U_m, \\
R_i (m ; \overbrace{1, \ldots,1, \underbrace{n}_{i\text{-th}} ,1, \ldots, 1}^m ) :~& U_{m+n-1} \rightarrow \mathbb{K}[U_n], \quad 1 \leq i \leq m
\end{align*}
so that the spaces $\{ \mathcal{O}(n) = \mathrm{Hom}(\mathbb{K}[U_n] \otimes A^{\otimes n}, A) \}_{n \geq 1}$ forms a non-symmetric operad with partial compositions
\begin{align*}
&(f \circ_i g)(r; a_1, \ldots, a_{m+n-1}) \\
&=  f (R_0 (m ; 1, \ldots, n , \ldots, 1 )r;~ a_1, \ldots, a_{i-1}, g (R_i (m ; 1, \ldots, n , \ldots, 1 ) r;~ a_i, \ldots, a_{i+n-1}), \ldots, a_{m+n-1}),
\end{align*}
for $f \in \mathcal{O}(m),$ $g \in \mathcal{O}(n)$ and $r \in U_{m+n-1}$. In case of dendriform algebras, $U_n = C_n$ and the functions $\{ R_0, R_i \}$ are given by (\ref{r-0}) and (\ref{r-i}). For other types of Loday-algebras, the corresponding sets $\{ U_n \}_{n \geq 1}$ are given in the table below, and the corresponding structure functions $\{ R_0, R_i \}$ are explicitly given in \cite{das-loday}.  Let `Lod' denotes either of the above types of Loday-algebras. Then a Lod-algebra structure on a vector space $A$ is equivalent to a multiplication $\pi$ on the above operad $\mathcal{O}$. The cohomology of the Lod-algebra $A$ is then defined by the cohomology induced from the multiplication $\pi$ in the operad $\mathcal{O}$. 

In all these Loday-algebras, there is a suitable set theoretic map $* : U_n \rightarrow U_n$ (for each $n \geq 1$) satisfying $* * = \mathrm{id}$. For various types of Loday-algebras mentioned above, the corresponding sets $\{ U_n \}_{n \geq 1}$ and the maps $* : U_n \rightarrow U_n$ are given by the following:

\medskip

\begin{center}
\begin{tabular}{|c|c|c|}
\hline
Type of algebras & $U_n$ & $* : U_n \rightarrow U_n$ \\ \hline \hline
associative dialgebras & $PBT_n $ (planar binary trees with $n$ vertices) & $* (t) = t^T$ (mirror reflection of $t$)\\ \hline
associative trialgebras & $PT_n$ (planar trees with $n$ vertices) & $* (t) = t^T$ (mirror reflection of $t$) \\ \hline
dendriform algebras & $C_n$ & $* ([r]) = [n-r+1]$ \\ \hline
dendriform trialgebras & $P_n =$ nonempty subsets of $\{1, \ldots, n \}$ &  $* (S) =\begin{cases} S^c ~ &\mathrm{~if ~} S \neq \{1,\ldots, n \} \\
S ~ &\mathrm{~if ~} S = \{1,\ldots, n \}  \end{cases} $ \\
\hline
quadri-algebras & $C_n \times C_n$ & $*([r],[r'])= (*([r]), *([r']) )$ \\ \hline
ennea-algebras & $ P_n \times P_n $ & $* (S, S') = (*(S), *(S')) $ \\ \hline
\end{tabular}
\end{center}

\medskip

\noindent Using the maps $*$ on $\{ U_n \}_{n \geq 1}$, we are now in a position to define involutive Lod-algebras. Let $A$ be a fixed Lod-algebra with structure functions $\{ R_0, R_i \}$ defined on sets $\{ U_n \}_{n \geq 1}$. Let $\pi \in \mathcal{O}(2) = \mathrm{Hom}( \mathbb{K}[U_2] \otimes A^{\otimes 2}, A)$ denote the corresponding multiplication. An involution on the Lod-algebra $A$ is given by a vector space involution $* : A \rightarrow A$ satisfying
\begin{align*}
\pi (r; a, b )^* = \pi ( *r; b^*, a^*),\text{ for } r \in U_2, a, b \in A.
\end{align*}

Let $(A, \pi, *)$ be an involutive Lod-algebra. Define
\begin{align*}
iC^n (A, A) = \{ f \in \mathcal{O}(n) |~ f (r; a_1, \ldots, a_n )^*  = (-1)^{\frac{(n-1)(n-2)}{2}}~ f (*r; a_n^*, \ldots, a_1^*),~ \forall r \in U_n \}.
\end{align*}
Then similar to Proposition \ref{inv-sub-comp}, one can show that $( \{ iC^n (A, A) \}_{n \geq 1}, \delta_\pi )$ is a subcomplex of the cochain complex $(\{ \mathcal{O}(n) \}_{n \geq 1}, \delta_\pi )$. The corresponding cohomology groups are called the cohomology of the involutive Lod-algebra $A$, and denoted by $iH^\bullet (A, A).$

Note that in this discussion we only consider (involutive) cohomology of Lod-algebras with coefficients in itself. The more general cohomology of Lod-algebras with coefficients in a representation is described in \cite{das-loday}. Using the approach of Subsection \ref{subsec-inv-cohomo}, we can define involutive representations of involutive Lod-algebras and their cohomology with coefficients in an involutive representation.

Moreover, similar to Theorem \ref{thm-decom}, one can show that for an involutive Lod-algebra $A$, the ordinary cohomology of the Lod-algebra $A$ splits as a direct sum of the involutive cohomology and a skew-factor. Since the proof is similar to Theorem \ref{thm-decom}, we do not repeat it here. 

\medskip

Given a quadratic Koszul operad $\mathcal{P}$, the notion of homotopy $\mathcal{P}$-algebra ($\mathcal{P}_\infty$-algebra in short) was given in \cite{loday-val} as an algebra over the Koszul resolution $\Omega \mathcal{P}^{\mathrm{i}}$, here $\mathcal{P}^{\mathrm{i}}$ is the Koszul dual of $\mathcal{P}$. However, sometimes it is hard to explicitly describe homotopy $\mathcal{P}$-algebras as the operad $\mathcal{P}^{\mathrm{i}}$ may be hard to explicitly write. For instance, the dual quadri operad is not very explicitly known. However, when a $\mathcal{P}$-algebra structure on a vector space $A$ is given by a multiplication in a non-symmetric operad $\{ \mathcal{O}(n) = \mathrm{Hom}(\mathbb{K}[U_n] \otimes A^{\otimes n }, A) \}$, the dual operad $\mathcal{P}^\mathrm{i}$ is already hidden in the description of $\mathcal{O}$. More precisely, it has been observed in \cite{das-loday} that $\mathcal{P}^\mathrm{i}(n) = \mathbb{K}[U_n]$, for $n \geq 1$. This relation between $\mathcal{P}^\mathrm{i}$ and the sets $\{ U_n \}_{n \geq 1}$ suggests to describe homotopy {Lod}-algebras as follows. This in particular recover $\mathrm{Dend}_\infty$-algebras when we consider the fixed Lod-algebra to be dendriform algebra.

\begin{defn}
A homotopy Lod-algebra is a graded vector space $\mathcal{A} = \oplus A_i$ together with a collection $\{  \mu_{k, r} : \mathcal{A}^{\otimes k} \rightarrow \mathcal{A}|~ 1 \leq k < \infty, r \in U_k \}$
of multilinear maps with $\mathrm{deg}(\mu_{k, r}) =k-2$ satisfying the following identities: for each $n \geq 1$ and $r \in U_n$,
\begin{align*}
\sum_{i+j= n+1} \sum_{\lambda = 1}^j~(-1)^{\lambda (i+1) + i (|a_1|+ \cdots + |a_{\lambda -1}|) } ~\mu_{j, R_0(j;1, \ldots, i, \ldots, 1)r} \big( a_1, \ldots, a_{\lambda -1},
\mu_{i, R_\lambda (j;1, \ldots, i, \ldots, 1)r} (a_\lambda, \ldots, a_{\lambda + i -1} ),\\ a_{\lambda + i}, \ldots, a_n \big) = 0, 
\end{align*}
for $a_1, \ldots, a_n \in \mathcal{A}.$
\end{defn}

An involutive homotopy {Lod}-algebra is a homotopy {Lod}-algebra $(\mathcal{A}, \mu_k)$ together with an involution $* : \mathcal{A} \rightarrow \mathcal{A},~ a \mapsto a^*$ on the underlying graded vector space $\mathcal{A}$ satisfying
\begin{align*}
\mu_{k, r}(a_1, \ldots, a_k)^*  = (-1)^\theta (-1)^{\frac{(n-1)(n-2)}{2}} ~\mu_{k, *r} (a_k^*, \ldots, a_1^*), ~ \text{ for } 1 \leq k < \infty ~ \text{ and } r \in U_k,
\end{align*}
where $\theta = \sum_{i=1}^k |a_i| (\sum_{j=i+1}^k |a_j| ).$

In the following sections, we mainly emphasize on oriented dendriform algebras, study their cohomology and deformations.
One may extend these results to other Lod-algebras equipped with orientations.

\section{Oriented dendriform algebras}\label{section-ori-dend}

Oriented associative algebras were introduced in \cite{koam-pira} as algebra with an action of an oriented group. In this section, we exhibit some further properties of oriented algebras and then introduce oriented dendriform algebras. 

\subsection{Oriented associative and Lie algebras}
Let $G$ be a group and $\varepsilon : G \rightarrow \{ \pm 1\}$ be a group homomorphism. A pair $(G, \varepsilon)$ is called an oriented group and $\varepsilon$ is called the orientation.

\begin{defn}
An oriented associative algebra over $(G, \varepsilon)$ is an associative algebra $A$ together with an action $G \times A \rightarrow A, ~(g, a) \mapsto ga$ satisfying
\begin{align*}
g (ab) = (ga) (gb) ~~~~ \text{ if } \varepsilon (g) = 1 \qquad  \text{ and } \qquad  g (ab) = (gb) (ga) ~~~~ \text{ if } \varepsilon (g) = -1.
\end{align*}
\end{defn}

Let $A$ and $B$ be oriented algebras over $(G, \varepsilon)$. A morphism between them is an algebra morphism $f : A \rightarrow B$ that is $G$-equivariant. We denote the category of oriented associative algebras over $(G, \varepsilon)$ by $\mathrm{\bf OAss}.$

Let $(G, \varepsilon)$ be an oriented group and $V$ be a vector space. Suppose there is an action of $G$ on $V$. Consider the tensor algebra $T(V) = \oplus_n V^{\otimes n}$ over $V$ with the concatenation product. Define an action of $G$ on $T(V)$ by
\begin{align*}
g (v_1 \otimes \cdots \otimes v_n ) = \begin{cases} gv_1 \otimes \cdots \otimes gv_n ~~~~ \text{ if } \varepsilon (g) = 1\\
gv_n \otimes \cdots \otimes gv_1 ~~~~ \text{ if } \varepsilon (g) = -1. \end{cases}
\end{align*}
Then $T(V)$ is an oriented associative algebra over $(G, \varepsilon).$

Let $V$ be a vector space equipped with a $G$-action. The free oriented associative algebra over $V$ is an oriented associative algebra $\mathcal{F}(V)$ over $(G, \varepsilon)$ equipped with a $G$-equivariant linear map $i : V \rightarrow \mathcal{F}(V)$ that satisfies the following universal condition:
for any oriented associative algebra $A$ over $(G, \varepsilon)$ and a $G$-equivariant linear map $f : V \rightarrow A$, there exists an unique oriented associative algebra morphism $\widetilde{f} : \mathcal{F}(V ) \rightarrow A$ such that $\widetilde{f} \circ i = f .$

\begin{prop}
Let $V$ be a vector space equipped with a $G$-action. Then $T(V)$ is a free oriented algebra over $V$.
\end{prop}

\begin{defn}
An oriented Lie algebra over $(G, \varepsilon)$ consists of a Lie algebra $(\mathfrak{g}, [~, ~])$ with an action $G \times \mathfrak{g} \rightarrow \mathfrak{g}, ~ ( g, x) \mapsto gx$ satisfying
\begin{align*}
g [ x,y] =  [gx, gy] ~~~~ \text{ if } \varepsilon(g) = 1  \qquad  \text{ and } \qquad
g[x,y] = [gy, gx]   ~~~~ \text{ if } \varepsilon(g) =- 1, ~~~ \text{ for } x, y \mathfrak{g}.
\end{align*}
\end{defn}

An oriented Lie algebras over $(G = \{ \pm 1 \}, \varepsilon = \mathrm{id})$ is given by a Lie algebra $\mathfrak{g}$ together with a Lie algebra anti-homomorphism.

A morphism between oriented Lie algebras over $(G, \varepsilon)$ can be defined similarly. We denote the category of oriented Lie algebras and morphisms between them by $\mathrm{\bf OLie}$.

Let $A$ be an oriented associative algebra over $(G, \varepsilon)$. Then it follows that the commutator Lie bracket $[a, b]_c = ab - ba$ on $A$ makes $A$ into an oriented Lie algebra over $(G, \varepsilon)$. This construction is functorial. Therefore, there is a functor $(~)_c : \mathrm{\bf OAss} \rightarrow \mathrm{\bf OLie}$ from the category of oriented associative algebras to the category of oriented Lie algebras. In the following, we show that the universal enveloping algebra of an oriented Lie algebra gives rise to a functor $U$ left adjoint to $(~)_c$.

Let $( \mathfrak{g}, [~, ~])$ be a Lie algebra. Consider the tensor algebra $T ( \mathfrak{g})$ of $\mathfrak{g}$. Then the universal enveloping algebra $U (\mathfrak{g})$ of $\mathfrak{g}$ is the quotient of $T (\mathfrak{g})$ by the two sided ideal generated by elements of the form $x \otimes y -  y \otimes x - [x, y]$, for $x, y \in \mathfrak{g}$. Let $( \mathfrak{g}, [~, ~])$ be an oriented Lie algebra over $(G, \varepsilon)$. Then the oriented associative algebra structure on $T ( \mathfrak{g})$ induces an orientation on $U ( \mathfrak{g})$ as 
\begin{align*}
 g (x \otimes y -  y \otimes x - [x, y]) = \begin{cases} gx \otimes gy -  gy \otimes gx - [gx, gy] \qquad \varepsilon (g) = 1 \\ gy \otimes gx -  gx \otimes gy - [gy, gx] \qquad \varepsilon (g) = -1. \end{cases}
\end{align*}
Therefore, $U(\mathfrak{g})$ is an oriented associative algebra with the induced orientation.

\begin{prop}
The functor $U : \mathrm{\bf OLie} \rightarrow \mathrm{\bf OAss}$ is left adjoint to the functor $(~)_c : \mathrm{\bf OAss} \rightarrow \mathrm{\bf OLie}$. In other words, there is an isomorphism
\begin{align*}
\mathrm{Hom}_{\mathrm{\bf OAss}} (U(\mathfrak{g}), A) \cong \mathrm{Hom}_{\mathrm{\bf OLie}} ( \mathfrak{g}, A_c),
\end{align*}
for any oriented associative algebra $A$ and oriented Lie algebra $\mathfrak{g}$ over $(G, \varepsilon).$
\end{prop}

\begin{proof}
For any oriented associative algebra morphism $f : U(\mathfrak{g}) \rightarrow A$, we consider its restriction to $\mathfrak{g}$. This is an oriented Lie algebra morphism $\mathfrak{g} \rightarrow A_c$.

Conversely, for any oriented Lie algebra morphism $h : \mathfrak{g} \rightarrow A_c$, we consider its unique extension as an oriented associative algebra morphism $\widetilde{h} : T (\mathfrak{g}) \rightarrow A$. It induces a map of oriented associative algebras $\widetilde{\widetilde{h}} : U (\mathfrak{g}) \rightarrow A$ as $h$ is a morphism of oriented Lie algebras.

The above two correspondences are inverses to each other. Hence the proof.
\end{proof}

\subsection{Oriented dendriform algebras}
Here we introduce an oriented version of dendriform algebras and show that Rota-Baxter operator on oriented associative algebras gives rise to oriented dendriform algebras.

\begin{defn}
An oriented dendriform algebra over $(G, \varepsilon)$ is a dendriform algebra $(D, \prec, \succ)$ together with an action $G \times D \rightarrow D, ~(g, a) \mapsto ga$ satisfying
\begin{align*}
 g ( a \prec b) = \begin{cases} ga \prec gb \quad \varepsilon (g) = + 1 \\
 gb \succ ga \quad \varepsilon (g) = - 1, \end{cases} \quad  g ( a \succ b) = \begin{cases} ga \succ gb \quad \varepsilon (g) = + 1 \\
 gb \prec ga \quad \varepsilon (g) = - 1. \end{cases}
\end{align*}
\end{defn}

Let $D$ and $D'$ be two oriented dendriform algebras over $(G, \varepsilon)$. A morphism between them consists of a dendriform algebra morphism $\psi : D \rightarrow D'$ (i.e. $D (a \prec b) = D(a) \prec D(b)$ and $D (a \succ b) = D(a) \succ D(b)$) satisfying $\psi (g a) = g \psi (a)$, for all $a , b \in D$ and $g \in G$.

\begin{exam}
\begin{itemize}
\item[(i)] When the oriented group $(G, \varepsilon)$ is given by the trivial orientation (i.e. $\varepsilon(G) = +1$), an oriented dendriform algebra over $(G, \varepsilon)$ is nothing but a dendriform algebra equipped with an action of $G$ by algebra automorphisms. Such dendriform algebras are called $G$-equivariant dendriform algebras in \cite{das-saha}.
\item[(ii)] For $G = \{ \pm 1 \}$ and $\varepsilon = \mathrm{id},$ an oriented dendriform algebra over $(G, \varepsilon)$ is a dendriform algebra $(D, \prec, \succ)$ together with an action by $-1$ whose square is identity. Thus, a dendriform algebra over $(G, \varepsilon)$ is nothing but an involutive dendriform algebra.
\item[(iii)] (Oriented MAX dendriform algebra) MAX dendriform structures appeared in noncommutative generalizations of the algebra of symmetric functions \cite{nov-thibon}. Let $X$ be an ordered set and let $X^\bullet$ be the set of all finite, non-empty words $x_1 \cdots x_m$ with alphabets $x_i$'s in $X$. For any word $a = x_1 \cdots x_m$, we denote by max$(a)$ the highest alphabet in $a$. Consider the free vector space $\mathbb{K}X^\bullet$ generated by $X^\bullet$. The MAX dendriform structure on $\mathbb{K}X^\bullet$ is defined on the basis as follows
\begin{align*}
a \prec b = \begin{cases} a \cdot b ~~ &\text{ if } \mathrm{max}(a) \geq \mathrm{max}(b) \\
0 ~~ &\text{ otherwise } \end{cases} \qquad
a \succ b = \begin{cases} a \cdot b ~~ &\text{ if } \mathrm{max}(a) < \mathrm{max}(b) \\
0 ~~ &\text{ otherwise}. \end{cases}
\end{align*} 
Let $(G, \varepsilon)$ be an oriented group. Suppose there is an action of $G$ on the oriented set $X$ that satisfies the followings: $x < y \Rightarrow  gx < gy$ if $\varepsilon (g) = +1$ and $x < y \Rightarrow  gx > gy$ if $\varepsilon (g) = -1$. The action of $G$ on $X$ extends to an action of $G$ on $\mathbb{K}X^\bullet$ defined on basis elements by $g(x_1 \cdots x_m ) = g(x_1) \cdots g(x_m)$ if $\varepsilon (g) = +1$ and $g(x_1 \cdots x_m ) = g(x_m) \cdots g(x_1)$ if $\varepsilon (g) = -1$. Then it is easy to see that the MAX dendriform algebra  $\mathbb{K}X^\bullet$ is an oriented dendriform algebra.
\end{itemize}
\end{exam}

Let $Y_n$ be the set of planar binary trees with $(n+1)$ leaves. Consider the space $\oplus_{n \geq 1} \mathbb{K}[Y_n]$ with the following two operations $\prec $ and $\succ$ defined recursively by
\begin{itemize}
\item $t \prec | = t $ and $| \prec t = 0$, for deg$(t) \geq 1$,
\item $t \succ | = 0$ and $| \succ t = t$, for  deg$(t) \geq 1$,
\item for $t = t_1 \vee t_2$ and $w = w_1 \vee w_2$ as grafting of trees,
\begin{align*}
t \prec w =~& t_1 \vee (t_2 \prec w + t_2 \succ w),\\
t \succ w =~& (t \prec w_1 + t \succ w_1) \vee w_2.
\end{align*}
For any vector space $V$, we define $\mathrm{Dend}(V) = \oplus_{n \geq 1} \mathbb{K}[Y_n] \otimes V^{ \otimes n }$ with the operations
\begin{align*}
(t; v_1 \otimes \cdots \otimes v_m) \prec (w; v_{m+1} \otimes \cdots \otimes v_{m+n}) = ( t \prec w; v_1 \otimes \cdots \otimes v_{m+n}),\\
(t; v_1 \otimes \cdots \otimes v_m) \succ (w; v_{m+1} \otimes \cdots \otimes v_{m+n}) = ( t \succ w; v_1 \otimes \cdots \otimes v_{m+n}).
\end{align*}
\end{itemize}
It is proved in \cite{loday} that $( \mathrm{Dend}(V), \prec, \succ )$ is a dendriform algebra. It is called the free dendriform algebra generated by $V$. 

Let $(G, \varepsilon)$ be an oriented group and there is an action of $G$ on $V$. Define an action of $G$ on $ \mathrm{Dend}(V)$ by
\begin{align*}
g (t ; v_1 \otimes \cdots \otimes v_m) = \begin{cases} (t; gv_1 \otimes \cdots \otimes gv_m)  \qquad \text{ if } \varepsilon(g) = 1 \\
(t^T ; gv_m \otimes \cdots \otimes gv_1 ) ~~~~ \quad \text{ if } \varepsilon(g) = - 1 \end{cases}
\end{align*}
where $t^T$ is the mirror reflection of $t$. With these notations, we have the following.

\begin{prop}
$\mathrm{Dend}(V)$ is an oriented dendriform algebra over $(G, \varepsilon).$
\end{prop}

\begin{proof}
For $\varepsilon (g) =1$, it is easy to see that the linear map on $\mathrm{Dend}(V)$ induced by $g$ is a dendriform algebra morphism. For $\varepsilon (g) = -1$, we have
\begin{align*}
g ~&((t; v_1 \otimes \cdots \otimes v_m) \prec (w; v_{m+1} \otimes \cdots \otimes v_{m+n}))  \\
=~& g ( t \prec w; v_1 \otimes \cdots \otimes v_{m+n}) \\
=~& ((t \prec w)^T ; gv_{m+n} \otimes \cdots \otimes gv_1) \\
=~& ((t_1 \vee (t_2 \prec w + t_2 \succ w))^T ; gv_{m+n} \otimes \cdots \otimes gv_1) \\
=~& ( (t_2 \prec w + t_2 \succ w)^T  \vee (t_1)^T ; gv_{m+n} \otimes \cdots \otimes g v_1) \\
=~& ((w^T \prec (t_2)^T + w^T \succ (t_2)^T) \vee (t_1)^T ; g v_{m+n} \otimes \cdots \otimes g v_1) \\
=~& ((w^T \prec (t^T)_1 + w^T \succ (t^T)_1) \vee (t^T)_2  ; g v_{m+n} \otimes \cdots \otimes g v_1) \qquad (\text{as } (t_1)^T = (t^T)_2 \text{ and } (t_2)^T = (t^T)_1)\\
=~& (w^T; gv_{m+n} \otimes \cdots \otimes gv_{m+1}) \succ (t^T; gv_m \otimes \cdots \otimes gv_{1})\\
=~& g (w; v_{m+1} \otimes \cdots \otimes v_{m+n}) \succ  g (t; v_1 \otimes \cdots \otimes v_m).
\end{align*}
Similarly, for $\varepsilon(g) = -1$, we can prove that
\begin{align*}
g ((t; v_1 \otimes \cdots \otimes v_m) \succ (w; v_{m+1} \otimes \cdots \otimes v_{m+n})) = g (w; v_{m+1} \otimes \cdots \otimes v_{m+n}) \prec  g (t; v_1 \otimes \cdots \otimes v_m).
\end{align*}
Hence the proof.
\end{proof}

\begin{defn}
A tridendriform algebra is a vector space $D$ together with three linear maps $\prec, \succ, \centerdot : D\otimes D \rightarrow D$ satisfying the following seven identities
\begin{align*}
& (a \prec b) \prec c =  a \prec (b \prec c + b \succ c + b \centerdot c), \qquad (a \succ b) \centerdot c =  a \succ (b \centerdot c), \\
& (a \succ b) \prec c = a \succ (b \prec c), ~~ \qquad \qquad \qquad \qquad  (a \prec b) \centerdot c = a \centerdot (b \succ c),  \\
& (a \prec b + a \succ b + a \centerdot b) \succ c = a \succ (b \succ c), \qquad (a \centerdot b) \prec c = a \centerdot (b \prec c),\\
& \qquad \qquad \qquad \qquad \qquad (a \centerdot b) \centerdot c =  a \centerdot (b \centerdot c),
\end{align*}
for all $ a, b, c \in D.$
\end{defn}

Let $(D, \prec, \succ, \centerdot)$ be a tridendriform algebra. Then $(D, \prec +~ \centerdot, \succ)$ is a dendriform algebra. Hence, the sum $* = \prec + \succ +~ \centerdot $ is an associative product on $D$.

\begin{defn}
An oriented tridendriform algebra over $(G , \varepsilon)$ is a tridendriform algebra $(D, \prec, \succ, \centerdot)$ together with an action $G \times D \rightarrow D$ satisfying
\begin{align*}
g ( a \prec b ) = \begin{cases} ga \prec gb \quad \varepsilon (g) =1 \\
gb \succ ga \quad \varepsilon (g) = - 1, \end{cases}  g ( a \succ b ) = \begin{cases} ga \succ gb \quad \varepsilon (g) =1 \\
gb \prec ga \quad \varepsilon (g) = - 1, \end{cases} g ( a \centerdot b ) = \begin{cases} ga \centerdot gb \quad \varepsilon (g) =1 \\
gb \centerdot ga \quad \varepsilon (g) = - 1. \end{cases}
\end{align*}
\end{defn}

\medskip



Recall that a linear map $R : A \rightarrow A$ on an (ordinary) associative algebra $A$ is said to be a Rota-Baxter operator on $A$ of weight $\lambda \in \mathbb{K}$ if $R$ satisfies
\begin{align*}
R(a) R(b) = R ( R(a) b + a R(b) + \lambda a b), ~~~ \text{ for all } a, b \in A.
\end{align*}
If $A$ is an oriented associative algebra over $(G, \varepsilon)$, then $R$ is a Rota-Baxter operator on oriented associative algebra $A$ if $R$ is further $G$-equivariant, i.e. $ R (ga) = g (R(a))$.


\begin{lemma}
Let $R$ be a Rota-Baxter operator of weight $\lambda$ on an oriented associative algebra $A$ over $(G, \varepsilon)$. Then
\begin{itemize}
\item[(i)] $- \lambda \mathrm{id} - R$ is a Rota-Baxter operator on $A$ of same weight $\lambda$.
\item[(ii)] For any scalar $\mu \in \mathbb{K}$, $\mu R$ is a Rota-Baxter operator on $A$ of weight $\mu \lambda$.
\end{itemize}
\end{lemma}

Rota-Baxter operator on an associative algebra induces dendriform and tridendriform algebra structure \cite{aguiar,fa}. The same result holds in the oriented context.

\begin{prop}
Let $R$ be a Rota-Baxter operator on an oriented associative algebra $A$ over $(G, \varepsilon)$. Then the following structure maps
\begin{align*}
a \prec b = a R(b) \qquad a \succ b = R(a) b \qquad a \centerdot b = \lambda ab, ~~~ \text{ for } a, b \in A
\end{align*}
make $A$ into an oriented tridendriform algebra over $(G, \varepsilon)$. If $\lambda = 0$ then $a \prec b = a R(b)$ and $a \succ b = R(a) b$ make $A$ into an oriented dendriform algebra over $(G, \varepsilon).$
\end{prop}

\begin{proof}
It is known that the above structure maps make $A$ into a tridendriform algebra \cite{aguiar}. It is also easy to verify that the above structure maps satisfy the oriented condition.
\end{proof}

Rota-Baxter operator on associative algebras arises from solutions of the associative classical Yang-Baxter equation (CYBE) \cite{aguiar2}. An element $r = \sum r_{(1)} \otimes r_{(2)} \in A \otimes A$ is a solution of the associative CYBE on an associative algebra $A$ if $r$ satisfies
\begin{align*}
r_{13} r_{12} - r_{12} r_{23} + r_{23} r_{13} = 0,
\end{align*}
where $r_{12} =  \sum r_{(1)} \otimes r_{(2)} \otimes 1$, $r_{13} = \sum r_{(1)} \otimes 1 \otimes r_{(2)}$ and $r_{23} = \sum 1 \otimes r_{(1)} \otimes r_{(2)} \in A \otimes A \otimes A$.
In such a case, the map $R : A \rightarrow A$ defined by $R(a) = \sum r_{(1)} a r_{(2)}$ is a Rota-Baxter operator on $A$ of weight $0$. In particular, if $A$ is an oriented associative algebra over $(G, \varepsilon)$, the solution $r$ is symmetric (i.e. $\sum r_{(1)} \otimes r_{(2)} = \sum r_{(2)} \otimes r_{(1)}$) and $G$-invariant (i.e. $r = (g \otimes g ) r$, for all $g \in G$), then the Rota-Baxter operator $R$ is $G$-equivariant. 
\medskip

Dendriform algebras are also related to pre-Lie structures \cite{aguiar,aguiar2}. It also holds in the oriented case.

\begin{defn}
A pre-Lie algebra consists of a vector space $L$ together with a linear map $\diamond : L \otimes L \rightarrow L$ satisfying
\begin{align*}
( a \diamond b ) \diamond c  -  a \diamond (b  \diamond c ) = ( b \diamond a ) \diamond c  -  b \diamond (a  \diamond c ), ~~~~ \text{ for } a, b, c \in L.
\end{align*}
\end{defn}

An oriented pre-Lie algebra over $(G, \varepsilon)$ is a pre-Lie algebra $(L, \diamond)$ together with an action $G \times L \rightarrow L, ~(g, x) \mapsto gx$ satisfying
\begin{align*}
g ( a \diamond b) =  ga \diamond gb \quad \text{ if } \varepsilon(g) = 1  ~~~~  \text{ and } ~~~~
g ( a \diamond b) = gb \diamond ga  \quad \text{ if } \varepsilon(g) =- 1.
\end{align*}

\begin{prop}
Let $(D, \prec, \succ )$ be an oriented dendriform algebra over $(G, \varepsilon)$. Then the operations
\begin{itemize}
\item[(i)] $a \diamond b =  a \succ b - b \prec a$ makes $D$ into an oriented pre-Lie algebra over $(G, \varepsilon)$;
\item[(ii)] $[a, b] = a \diamond b - b \diamond a$ makes $D$ into an oriented Lie algebra over $(G, \varepsilon)$.
\end{itemize}
\end{prop}

\section{Cohomology of oriented dendriform algebras}\label{section-cohomo-ori-dend}

In this section, we introduce cohomology for oriented dendriform algebras. This cohomology is obtained from a bicomplex that mixes the cochain complex computing group cohomology and the cochain complex computing dendriform algebra cohomology.

Let $(G, \varepsilon)$ be a fixed oriented group and $D$ be an oriented dendriform algebra over $(G, \varepsilon)$. An oriented representation of $D$ consists of an ordinary representation $M$ together with a $G$-module structure on $M$ satisfying
\begin{align*}
g ( a \prec m ) = \begin{cases} ga \prec gm \quad \varepsilon(g) =+1 \\ gm \succ ga \quad \varepsilon(g) =-1, \end{cases} \qquad g ( a \succ m ) = \begin{cases} ga \succ gm \quad \varepsilon(g) =+1 \\ gm \prec ga \quad \varepsilon(g) =-1, \end{cases}  \\
g ( m \prec a ) = \begin{cases} gm \prec ga \quad \varepsilon(g) =+1 \\ ga \succ gm \quad \varepsilon(g) =-1, \end{cases} \qquad g ( m \succ a ) = \begin{cases} gm \succ ga \quad \varepsilon(g) =+1 \\ ga \prec gm \quad \varepsilon(g) =-1. \end{cases}
\end{align*}

Let $D$ be an oriented dendriform algebra over $(G, \varepsilon)$ and $M$ an oriented representation of $D$. For each $n \geq 1$, the group $G$ acts on the $n$-th cochain group $C^n_{\mathrm{dend}} (D, M) = \mathrm{Hom}( \mathbb{K}[C_n] \otimes D^{\otimes n}, M)$ by
\begin{align}\label{cochain-action}
(gf)([r]; a_1, \ldots, a_n) = \begin{cases} g f ( [r]; g^{-1}a_1, \ldots, g^{-1}a_n) ~~~\hspace{.3cm} \qquad \qquad \qquad  &\text{ if } \varepsilon(g) = + 1 \\ (-1)^{\frac{(n-1)(n-2)}{2}}~ gf([n-r+1]; g^{-1} a_n, \ldots, g^{-1} a_1 ) \quad &\text{ if } \varepsilon(g) =-1, \end{cases}
\end{align}
for $[r] \in C_n$ and $a_1, \ldots, a_n \in D$. With this notation, we have the following.

\begin{prop}
The dendriform algebra cochain complex $( \{ C^n_{\mathrm{dend}} (D, M) \}_{n \geq 1}, \delta_{\mathrm{dend}})$ is a $G$-equivariant cochain complex.
\end{prop}

\begin{proof}
We have to prove that the coboundary map $\delta_{\mathrm{dend}}$ is $G$-equivariant map with respect to the above $G$-action on cochain groups. Let $f \in C^n_{\mathrm{dend}}(D,M)$ and $g \in G$ with $\varepsilon (g) = + 1$. Then for any $[r] \in C_{n+1}$ and $a_1, \ldots, a_{n+1} \in D$, we have
\begin{align*}
& \delta_{\mathrm{dend}} (gf) ([r]; a_1 , \ldots, a_{n+1}) \\
&=  \theta_1 \big( R_0 (2;1,n) [r]; ~a_1, (gf) (R_2 (2;1,n)[r]; a_2, \ldots, a_{n+1})   \big) \\
&+ \sum_{i=1}^n  (-1)^i~ (gf) \big(  R_0 (n; 1, \ldots, 2, \ldots, 1)[r]; a_1, \ldots, a_{i-1}, \pi_D (R_i (n;1, \ldots, 2, \ldots, 1)[r]; a_i, a_{i+1}), \ldots, a_{n+1}   \big) \\
&+ (-1)^{n+1} ~\theta_2 \big( R_0 (2; n, 1) [r];~ (gf) (R_1 (2;n,1)[r]; a_1, \ldots, a_n), a_{n+1}   \big) \\
& =\theta_1 \big( R_0 (2;1,n) [r]; ~a_1, g (f (R_2 (2;1,n)[r]; g^{-1}a_2, \ldots, g^{-1}a_{n+1})) \big) \\
&+ \sum_{i=1}^n  (-1)^i~ g(f \big(  R_0 (n; 1, \ldots, 2, \ldots, 1)[r]; g^{-1}a_1, \ldots, g^{-1}a_{i-1},\\& \qquad \qquad \qquad \qquad \qquad \pi_D (R_i (n;1, \ldots, 2, \ldots, 1)[r]; g^{-1}a_i, g^{-1}a_{i+1}), \ldots, g^{-1}a_{n+1}   )\big) \\
&+ (-1)^{n+1} ~\theta_2 \big( R_0 (2; n, 1) [r];~ g(f (R_1 (2;n,1)[r]; g^{-1}a_1, \ldots, g^{-1}a_n)), a_{n+1}  \big).
\end{align*}
On the other hand
\begin{align*}
&(g(\delta_{\mathrm{dend}} f)) ([r]; a_1 , \ldots, a_{n+1}) \\
&= g \big( (\delta_{\mathrm{dend}} f) ([r]; g^{-1}a_1 , \ldots, g^{-1}a_{n+1}) \big)    \\
&=  g \big( \theta_1 ( R_0 (2;1,n) [r]; ~g^{-1}a_1, f (R_2 (2;1,n)[r]; g^{-1}a_2, \ldots, g^{-1}a_{n+1})  ) \big)  \\
&+ \sum_{i=1}^n  (-1)^i~ g \big( f (  R_0 (n; 1, \ldots, 2, \ldots, 1)[r]; g^{-1}a_1, \ldots, g^{-1}a_{i-1}, \\
& \qquad \qquad \qquad \qquad \qquad \pi_D (R_i (n;1, \ldots, 2, \ldots, 1)[r]; g^{-1}a_i, g^{-1}a_{i+1}), \ldots, g^{-1}a_{n+1} )  \big) \\
&+ (-1)^{n+1} ~g \big( \theta_2 ( R_0 (2; n, 1) [r];~ f (R_1 (2;n,1)[r]; g^{-1}a_1, \ldots, g^{-1}a_n), g^{-1}a_{n+1}   )  \big) \\
&=  \theta_1 \big( R_0 (2;1,n) [r]; ~ a_1, g(f (R_2 (2;1,n)[r]; g^{-1}a_2, \ldots, g^{-1}a_{n+1}))   \big) \\
&+ \sum_{i=1}^n  (-1)^i~ g \big( f (  R_0 (n; 1, \ldots, 2, \ldots, 1)[r]; g^{-1}a_1, \ldots, g^{-1}a_{i-1}, \\
& \qquad \qquad \qquad \qquad \qquad \pi_D (R_i (n;1, \ldots, 2, \ldots, 1)[r]; g^{-1}a_i, g^{-1}a_{i+1}), \ldots, g^{-1}a_{n+1} )  \big) \\
&+ (-1)^{n+1} ~\theta_2 \big( R_0 (2; n, 1) [r];~ g(f (R_1 (2;n,1)[r]; g^{-1}a_1, \ldots, g^{-1}a_n)), a_{n+1}   \big).
\end{align*}
Thus, for $\varepsilon(g) = + 1$, we get $\delta_{\mathrm{dend}}(gf) = g (\delta_{\mathrm{dend}}(f))$. Next we verify the same when $\varepsilon(g) = - 1$. In this case, we observe that
\begin{align*}
&\delta_{\mathrm{dend}} (gf) ([r]; a_1 , \ldots, a_{n+1}) \\
&=  \theta_1 \big( R_0 (2;1,n) [r]; ~a_1, (gf) (R_2 (2;1,n)[r]; a_2, \ldots, a_{n+1})   \big) \\
&+ \sum_{i=1}^n  (-1)^i~ (gf) \big(  R_0 (n; 1, \ldots, 2, \ldots, 1)[r]; a_1, \ldots, a_{i-1},
\pi_D (R_i (n;1, \ldots, 2, \ldots, 1)[r]; a_i, a_{i+1}), \ldots, a_{n+1}   \big) \\
&+ (-1)^{n+1} ~\theta_2 \big( R_0 (2; n, 1) [r];~ (gf) (R_1 (2;n,1)[r]; a_1, \ldots, a_n), a_{n+1}   \big) \\
& =(-1)^{\frac{(n-1)(n-2)}{2}}~ \theta_1 \big( R_0 (2;1,n) [r]; ~a_1, g (f (R_2 (2;n,1)[n-r+2]; g^{-1}a_{n+1}, \ldots, g^{-1}a_{2}) \big) \\
&+ (-1)^{\frac{(n-1)(n-2)}{2}} (-1)^{n+1} \sum_{i=1}^{n} (-1)^{n-i+1} ~ g(f \big(  R_0 (n; 1, \ldots, \underbrace{2}_{n-i+1}, \ldots, 1)[n-r+2]; g^{-1}a_{n+1}, \ldots, g^{-1}a_{i+2}, \\
& \qquad \qquad \qquad \pi_D (R_{n-i+1} (n; 1, \ldots, \underbrace{2}_{n-i+1}, \ldots, 1)[n-r+2]; g^{-1}a_{i+1}, g^{-1}a_{i}), g^{-1}a_{i-1}, \ldots, g^{-1}a_{1})\big) \\
&+ (-1)^{\frac{(n-1)(n-2)}{2} + (n+1)} ~\theta_2 \big( R_0 (2; n, 1) [r];~ g(f (R_2 (2;1,n)[n-r+2]; g^{-1}a_n, \ldots, g^{-1}a_1)), a_{n+1} \big).
\end{align*}
We also have
\begin{align*}
&(g(\delta_{\mathrm{dend}} f)) ([r]; a_1 , \ldots, a_{n+1}) \\
&= (-1)^{\frac{n(n-1)}{2}} g ((\delta_{\mathrm{dend}} f) ([n-r+2]; g^{-1}a_{n+1} , \ldots, g^{-1}a_{1}))   \\
&= (-1)^{\frac{n(n-1)}{2}} g \big( \theta_1 ( R_0 (2;1,n) [n-r+2]; ~g^{-1}a_{n+1}, f (R_2 (2;1,n)[n-r+2]; g^{-1}a_2, \ldots, g^{-1}a_{1})  ) \big) \\
& + (-1)^{\frac{n(n-1)}{2}}\sum_{i=1}^n (-1)^{n-i+1} ~ g \big( f (  R_0 (n; 1, \ldots, \underbrace{2}_{n-i+1}, \ldots, 1)[n-r+2]; g^{-1}a_{n+1}, \ldots, g^{-1}a_{i+2}, \\
& \qquad \qquad \qquad \pi_D (R_{n-i+1} (n;1, \ldots, \underbrace{2}_{n-i+1}, \ldots, 1)[n-r+2]; g^{-1}a_{i+1}, g^{-1}a_{i}), g^{-1}a_{i-1},\ldots, g^{-1}a_{1}  ) \big) \\
&+ (-1)^{\frac{n(n-1)}{2} + (n+1)} ~g \big( \theta_2 ( R_0 (2; n, 1) [n-r+2];~ f (R_1 (2;n,1)[n-r+2]; g^{-1}a_{n+1}, \ldots, g^{-1}a_2), g^{-1}a_{1}  ) \big) \\
&=  (-1)^{\frac{n(n-1)}{2}}~ \theta_2 \big( R_0 (2;n,1) [r]; ~ g(f (R_1 (2;1,n)[n-r+2]; g^{-1}a_n, \ldots, g^{-1}a_{1})), a_{n+1} \big) \\
& + (-1)^{\frac{n(n-1)}{2}}\sum_{i=1}^n (-1)^{n-i+1} ~ g \big( f (  R_0 (n; 1, \ldots, \underbrace{2}_{n-i+1}, \ldots, 1)[n-r+2]; g^{-1}a_{n+1}, \ldots, g^{-1}a_{i+2}, \\
& \qquad \qquad \qquad \pi_D (R_{n-i+1} (n;1, \ldots, \underbrace{2}_{n-i+1}, \ldots, 1)[n-r+2]; g^{-1}a_{i+1}, g^{-1}a_{i}), g^{-1}a_{i-1},\ldots, g^{-1}a_{1}  ) \big) \\
&+ (-1)^{\frac{n(n-1)}{2}} (-1)^{n+1} ~ \theta_1 \big( R_0 (2;1, n) [r];~a_{1}, g(f (R_2 (2;n,1)[n-r+2]; g^{-1}a_{n+1}, \ldots, g^{-1}a_2))  \big).
\end{align*}
Thus, for $\varepsilon(g) = -1$ we also get $\delta_{\mathrm{dend}}(gf) = g(\delta_{\mathrm{dend}}(f))$. Hence the proof.
\end{proof}

It follows from the above proposition and Remark \ref{rem-grp-coho} that we obtain a bicomplex 
\begin{align}\label{bicom}
\xymatrix{
\vdots & \vdots  &  \\
 \mathrm{Maps} (G^{\times 2}, \mathrm{Hom}(\mathbb{K}[C_1] \otimes D, M)) \ar[r]^{\partial_h} \ar[u]^{\partial_v} & \mathrm{Maps} (G^{\times 2}, \mathrm{Hom}(\mathbb{K}[C_2] \otimes D^{\otimes 2}, M)) \ar[r]^{}  \ar[u]^{\partial_v} & ~ \cdots\\
 \mathrm{Maps} (G , \mathrm{Hom}(\mathbb{K}[C_1] \otimes D, M) ) \ar[u]^{\partial_v} \ar[r]^{\partial_h} & \mathrm{Maps} (G, \mathrm{Hom}(\mathbb{K}[C_2] \otimes D^{\otimes 2}, M)) \ar[u]^{\partial_v} \ar[r]^{} & ~ \cdots\\
   \mathrm{Hom} (\mathbb{K}[C_1] \otimes D, M) \ar[r]^{\partial_h} \ar[u]^{\partial_v} &  \mathrm{Hom}(\mathbb{K}[C_2] \otimes D^{\otimes 2}, M)  \ar[u]^{\partial_v} \ar[r]^{\partial_h} & \cdots .
}
\end{align}
The coboundary maps are explicitly given as follows.

\medskip

\noindent (A) The horizontal coboundary maps $\partial_h$ are given by
\begin{align*}
&(\partial_{h} f) (g_1,\ldots, g_n;[r]; a_1 , \ldots, a_{n+1}) \\
&=  \theta_1 \big( R_0 (2;1,n) [r]; ~a_1, f (g_1,\dots,g_n;R_2 (2;1,n)[r]; a_2, \ldots, a_{n+1})   \big) \\
&+ \sum_{i=1}^n  (-1)^i~ f \big(g_1,\ldots,g_n;  R_0 (n; 1, \ldots, 2, \ldots, 1)[r]; a_1, \ldots, a_{i-1}, \pi_D (R_i (1, \ldots, 2, \ldots, 1)[r]; a_i, a_{i+1}), \ldots, a_{n+1}   \big) \\
&+ (-1)^{n+1} ~\theta_2 \big( R_0 (2; n, 1) [r];~ f (g_1,\ldots,g_n;R_1 (2;n,1)[r]; a_1, \ldots, a_n), a_{n+1}   \big), ~~~ \text{ for } [r] \in C_{n+1}.
\end{align*}

\medskip

\noindent (B) The coboundary of the first vertical maps $\partial_v$ is given by
\begin{align*}
&(\partial_{v} \beta) (g_1,\ldots, g_{n+1};[1];a) \\
&=  g_1 \beta(g_2,\ldots, g_{n+1};[1];g_1^{-1}a)
+ \sum_{i=1}^n  (-1)^i~ \beta \big(g_1,\ldots,g_i g_{i+1},\ldots,g_{n+1};[1];a) 
+ (-1)^{n+1} ~ \beta (g_1,\ldots,g_n;[1];a).
\end{align*}

\medskip

\noindent (C) The coboundary of the second vertical maps $\partial_v$ is given by
\begin{align*}
&(\partial_{v} \gamma) (g_1,\ldots, g_{n+1};[r];a,b) \\
&= \begin{cases}  g_1 \gamma(g_2,\ldots, g_{n+1};[r];g_1^{-1}a, g_1^{-1}b) + \sum_{i=1}^n  (-1)^i~ \gamma \big(g_1,\ldots,g_i g_{i+1},\ldots,g_{n+1};[r];a,b) \\ \qquad \qquad \qquad + (-1)^{n+1} ~ \gamma (g_1,\ldots,g_n;[r];a,b) \quad \text{if } \varepsilon(g_1) = +1 \\\\
g_1 \gamma(g_2,\ldots, g_{n+1};[3-r];g_1^{-1}b, g_1^{-1}a) + \sum_{i=1}^n  (-1)^i~ \gamma \big(g_1,\ldots,g_i g_{i+1},\ldots,g_{n+1};[r];a,b) \\ \qquad \qquad \qquad + (-1)^{n+1} ~ \gamma (g_1,\ldots,g_n;[r];a,b) \quad \text{if } \varepsilon(g_1) = -1 , \end{cases}
\end{align*}
for $[r] \in C_2$.


We define 
\begin{align*}
C^n_{G, \mathrm{dend}} (D, M) = \bigoplus_{i+j=n, i \geq 0, j \geq 1} \mathrm{Maps} (G^{\times i}, \mathrm{Hom}(\mathbb{K}[C_j] \otimes D^{\otimes j}, M)), ~ \text{ for } n \geq 1
\end{align*}
with the convension that  $\mathrm{Maps} (G^{0}, \mathrm{Hom}(\mathbb{K}[C_j] \otimes D^{\otimes j}, M)) = \mathrm{Hom}(\mathbb{K}[C_j] \otimes D^{\otimes j}, M)$. The coboundary map $\partial : C^n_{G, \mathrm{dend}} (D, M) \rightarrow C^{n+1}_{G, \mathrm{dend}} (D, M)$ is induced from the bicomplex (\ref{bicom}), i.e.
\begin{align*}
\partial ( \nu ) =  \partial_v (c) + (-1)^i \partial_h (c), ~ \text{ for } \nu \in \mathrm{Maps} (G^{\times i}, \mathrm{Hom}(\mathbb{K}[C_j] \otimes D^{\otimes j}, M)).
\end{align*}
We denote the space of $n$-cocycles by $Z^n_{G, \mathrm{dend}} (D,M)$ and the space of $n$-coboundaries by $B^n_{G, \mathrm{dend}}(D,M)$. Then $B^n_{G, \mathrm{dend}}(D,M) \subset Z^n_{G, \mathrm{dend}} (D,M) \subset C^n_{G, \mathrm{dend}} (D,M)$. The corresponding cohomology groups are denoted by ${H}^n_{G, \mathrm{dend}} (D, M)$, for $n \geq 1$ and called the cohomology of the oriented dendriform algebra. 

\medskip

\begin{remark}
Let us explain the second cohomology group ${H}^2_{G, \mathrm{dend}} (D, M)$ in details.
It follows from (A), (B) and (C) above that a pair 
\begin{align*}(\alpha,\beta) \in  \text{Maps}(G, \text{Hom}(\mathbb{K}[C_1]\otimes D,M)) \oplus \text{Hom}(\mathbb{K}[C_2]\otimes D^{\otimes 2},M)
\end{align*}
is a $2$-cocycle, that is, lies in $Z^2_{G, \mathrm{dend}} (D, M)$ if they satisfy 
\begin{align}\label{coho 1}
\alpha(gh ; a) = {g \alpha(h; {{g^{-1}}{a}})} + \alpha(g; a),
\end{align}
\begin{align}\label{2-cocycle exp}
\theta_1 \big( [r]; a, \alpha (g ; b)   \big)  -&  \alpha \big(g ; \pi_D ([r]; a, b) \big)  + \theta_2 \big( [r]; \alpha (g ;  a), b   \big) \nonumber \\
&= \begin{cases}
\,g(\beta([r];{g^{-1}}a,{{g^{-1}}{}b)}) - \beta([r];a, b)   &{\rm if}\ \varepsilon(g) = +1\\
\,g (\beta([3-r];{g^{-1}}{}b,{g^{-1}}{}a)) - \beta([r];a, b)  &{\rm if} \  \varepsilon(g) = -1, \end{cases}
\end{align}
and
\begin{align}\label{2-2co}
\partial_h (\beta ) := \delta_{\mathrm{dend}} (\beta) = 0,
\end{align}
for all $g, h \in G;~ a,b \in D$ and $[r]\in C_2$. Here we have used the notation $\alpha (g;a)$ in the place of $\alpha (g;[1]; a)$. 

\medskip

Moreover, the pair $(\alpha, \beta)$ is a $2$-coboundary, that is, lies in ${B}^2_{G, \mathrm{dend}}(D,M)$ if and only if  there exists $\gamma \in C^1_G (D, M) = \text{Hom}(\mathbb{K}[C_1]\otimes D, M) \simeq \text{Hom} (D,M)$ such that 
 \begin{align}\label{when-cob}
 \begin{cases} \alpha(g;a) =~ g (\gamma (g^{-1}a))- \gamma(a),\\
\beta ([r]; a, b) =~ \theta_0 ([r]; a, \gamma(b)) - \gamma(\pi_D([r]; a, b)) + \theta_2([r]; \gamma(a), b), ~ \text{ for } [r] \in C_2. \end{cases}
 \end{align}
 \end{remark}

 

Let us use the following notations:
\begin{align*}
\beta^l(a, b) = \beta([1]; a, b) ~\text{ and }~ \beta^r(a, b) = \beta([2]; a, b).
\end{align*}
Using these notations we can rewrite Equation (\ref{2-cocycle exp}) as follows:
\begin{align}\label{2-cocycle exp left}
a \prec \alpha (g ; b) - \alpha(g ; a \prec b)+ \alpha (g ;  a) \prec b 
= \begin{cases} 
\,g (\beta^l (g^{-1}a,g^{-1}b)) - \beta^l(a, b)  &  {\rm if}\ \varepsilon(g) = +1\\
\,g (\beta^r (g^{-1}b,g^{-1}a)) - \beta^l(a, b)   & {\rm if} \  \varepsilon(g) = -1, \end{cases}
\end{align}
and
\begin{align}\label{2-cocycle exp right}
a \succ \alpha (g; b) - \alpha(g; a \succ b)+ \alpha (g;  a) \succ b 
= \begin{cases}
\,g (\beta^r (g^{-1}a,g^{-1}b)) - \beta^r (a, b)  & {\rm if}\ \varepsilon(g) = +1\\
\,g (\beta^l (g^{-1}b,g^{-1}a)) - \beta^r (a, b)  & {\rm if} \  \varepsilon(g) = -1.\end{cases}
\end{align}

In the following, we show that the second cohomology group is closely related to extensions and deformations of oriented dendriform algebras.

\subsection{Relation with the cohomology of oriented associative algebras}
In Subsection \ref{subsec-inv-cohomo} we show that there is a morphism from the cohomology of an involutive dendriform algebra to the Hochschild cohomology of the corresponding involutive associative algebra. In this subsection, we prove a similar result in the oriented context.

We first recall the cohomology of an oriented associative algebra as defined in \cite{koam-pira}. Let $A$ be an oriented associative algebra over $(G, \varepsilon)$ and $M$ be an oriented $A$-bimodule. Consider the standard Hochschild cochain complex of $A$ with coefficients in $M$,
\begin{align*}
M \rightarrow \mathrm{Hom}(A, M) \rightarrow \mathrm{Hom}(A^{\otimes 2}, M) \rightarrow \mathrm{Hom}(A^{\otimes 3}, M) \rightarrow \cdots .
\end{align*}
The group $G$ acts on the above cochain groups (similar to (\ref{cochain-action})) and making the above complex a $G$-equivariant complex. The cohomology of the following bicomplex (obtained from the induced bicomplex by deleting the first column)
\begin{align*}
\xymatrix{
\vdots & \vdots  &  \\
 \mathrm{Maps} (G^{\times 2}, \mathrm{Hom}( A, M)) \ar[r] \ar[u] & \mathrm{Maps} (G^{\times 2}, \mathrm{Hom}( A^{\otimes 2}, M)) \ar[r]  \ar[u] & ~ \cdots\\
  \mathrm{Maps} (G , \mathrm{Hom}(A, M) ) \ar[u] \ar[r] & \mathrm{Maps} (G, \mathrm{Hom}( A^{\otimes 2}, M)) \ar[u] \ar[r] & ~ \cdots\\
   \mathrm{Hom} ( A, M) \ar[r] \ar[u] &  \mathrm{Hom}(A^{\otimes 2}, M)  \ar[u] \ar[r] & \cdots
}
\end{align*}
is called the cohomology of the oriented associative algebra $A$ with coefficients in $M$, and we denote them by $H_{G, \mathrm{Hoch}}^* (A, M).$

Let $D$ be an oriented dendriform algebra over $(G, \varepsilon)$ and $M$ be an oriented representation of $D$. Consider the corresponding oriented associative algebra structure on $D$ and oriented associative $D$-bimodule structure on $M$. It follows from \cite[Theorem 2.9]{das} that the collection of maps
\begin{align*}
S_n : \mathrm{Hom}(\mathbb{K}[C_n] \otimes D^{\otimes n}, M) ~&\rightarrow ~\mathrm{Hom}( D^{\otimes n}, M) , ~~~~ \text{ for } n \geq 1 \\
f ~&\mapsto ~ f_{[1]} + \cdots + f_{[n]}
\end{align*}
defines a morphism from the ordinary dendriform algebra cochain complex $( \{ C^n_{\mathrm{dend}} (D, M) \}_{n\geq 1}, \delta_{\mathrm{dend}})$ to the truncated Hochschild cochain complex $(\{ C^n_{\mathrm{Hoch}} (D, M)\}_{n\geq 1}, \delta_{\mathrm{Hoch}})$. It is easy to see that this morphism is $G$-equivariant. Hence it induces a map between the corresponding bicomplexes. Hence, we get the following.

\begin{prop}\label{ass-dend-mor}
The collection of maps $\{ S_n \}$ induces a morphism
\begin{align*}
H_{G, \mathrm{dend}}^* (D, M) \rightarrow H_{G, \mathrm{Hoch}}^* (D, M) 
\end{align*}
from the cohomology of the oriented dendriform algebra to the cohomology of the corresponding oriented associative algebra.
\end{prop}

\subsection{Extensions of oriented dendriform algebras}

It is known that the second Hochschild cohomology group of an associative algebra with coefficients in a bimodule classify equivalence classes of abelian extensions. This result has been extended to various other algebras, including dendriform algebras \cite{das}. In this subsection, we prove an oriented version of this result for dendriform algebras.

Let $D$ be an oriented dendriform algebra over $(G, \varepsilon)$ and $M$ be a $G$-module. Then $M$ can be considered as an oriented dendriform algebra over $(G, \varepsilon)$ with trivial dendriform structure.

\begin{defn}
An extension of $D$ by $M$ is a $\mathbb{K}$-split sequence of oriented dendriform algebras
\begin{align}\label{exten}
\xymatrix{
0 \ar[r] & M \ar[r]^{i} & B \ar[r]^p & D \ar[r] & 0.
}
\end{align}
\end{defn}

Let the splitting of (\ref{exten}) is given by a section $s : D \rightarrow B$ of the map $p$, i.e. $p \circ s = \mathrm{id}_D$. 
An extension of $D$ by $M$ induces an oriented representation of $D$ on $M$ given by
\begin{align*}
a \prec m = s(a) \prec i(m) \qquad a \succ m = s(a) \succ i(m) \qquad m \prec a = i(m) \prec s(a) \qquad m \succ a =  i(m) \succ s(a).
\end{align*} 
This representation does not depend on the chosen section $s$.
\begin{defn}
Two such extensions $B$ and $B'$ are said to equivalent if there is a morphism $\phi : B \rightarrow B'$ of oriented dendriform algebras making the following diagram commutative
\begin{align*}
\xymatrix{
0 \ar[r] & M \ar@{=}[d] \ar[r]^{i} & B \ar[r] \ar[d]^\phi \ar[r]^p & D \ar@{=}[d] \ar[r] & 0 \\
0 \ar[r] & M \ar[r]_{i} & B' \ar[r]_p & D \ar[r] & 0.
}
\end{align*}
\end{defn}

Next, let us fix an oriented representation $M$ of the oriented dendriform algebra $D$. We denote by $\mathcal{E}xt_G (D, M)$ the equivalence classes of extensions of $D$ by $M$ for which the induced representation on $M$ is the prescribed one. With these notations, we have the following.


%

\begin{thm}\label{ext-thm}
There is a one-to-one correspondence between $\mathcal{E}xt_G (D, M)$ and the second cohomology group ${H}^2_{G, \mathrm{dend}} (D, M).$
\end{thm}

\begin{proof}
Let  $(\alpha, \beta) \in Z^2_{G, \mathrm{dend}}(D,M)$ be a $2$-cocycle. Take
$B =M\oplus D.$
We define the $G$-action and multiplications on $B$ by
\begin{align*}
g (m , a) =~& (gm + \alpha(g; ga), ga),\\
(m, a) \prec (n,b) =~& (m \prec b + a \prec n + \beta^l(a, b), a \prec b), \\
(m, a) \succ (n, b) =~& (m \succ b + a \succ n + \beta^r (a, b), a \succ b),
\end{align*}
for $(m,a), (n, b) \in B$. We show that $B$ is an extension of the oriented dendriform algebra $D$ by $M$. Note that $B$ is a dendriform algebra as $\beta$ is a $2$-cocyle in the ordinary dendriform cohomology \cite{das}. Next, we observe that for $\varepsilon (g) = +1$,
\begin{align*}
g( (m,a) \prec (n,b)) & = g(m \prec b + a \prec n +\beta^l(a, b), a \prec b) \\
& = (g m \prec g b + g a \prec gn + g \beta^l (a,b) + \alpha(g; ga\prec gb), ga \prec gb).
\end{align*}
On the other hand, we have
\begin{align*}
g(m,a) \prec g(n,b) & = (gm + \alpha(g; ga), ga) \prec (gn + \alpha(g; gb), gb) \\
                                                &= \big( (gm + \alpha(g; ga)) \prec gb + ga \prec (gn + \alpha(g; gb)) + \beta^l (ga, gb), ga \prec gb  \big).
\end{align*}
Applying Equation (\ref{2-cocycle exp left}) in the above equation, we conclude that 
$$g ( (m,a) \prec (n,b)) = g(m,a) \prec g(n,b).$$
By a similar compuatation for $\varepsilon(g) =-1$, we get
$g ( (m,a) \prec (n,b)) = g(n,b) \succ g(m,a).$
In a similar way,
\begin{align*}
g ( (m,a) \succ (n,b)) = \begin{cases}  g(m,a) \succ g(n,b) & \mathrm{if }~~ \varepsilon(g) = +1\\
g(n,b) \prec g(m,a) & \mathrm{if }~~ \varepsilon(g) = -1. \end{cases}
\end{align*}
This shows that $B$ is an oriented dendriform algebra. Next, we define maps $i : M \to B = M\oplus D$ and $p :B = M\oplus D \to D $ by
\begin{align*}
i(m) = (m, 0) \quad \text{ and } \quad p(m,a) = a.
\end{align*}
It is easy to verify that $0 \rightarrow M \xrightarrow{i} B \xrightarrow{p} D \rightarrow 0$ is an extension of $D$ by $M$. 

Let $(\alpha', \beta') \in Z^2_{G, \mathrm{dend}} (D, M)$ be another $2$-cocycle and the difference $(\alpha, \beta ) - (\alpha', \beta') $ be a $2$-coboundary, i.e. in $B^2_{G, \mathrm{dend}} (D, M)$. Then there exists $\gamma \in \mathrm{Hom}(D, M)$ such that
\begin{align*}
&(\alpha - \alpha' )(g;a) = g ( \gamma (g^{-1}a)) - \gamma (a),\\
& (\beta - \beta')([r];a, b)= \theta_0 ( [r]; a, \gamma(b)) - \gamma ( \pi_D ([r];a, b)) + \theta_2 ([r]; \gamma (a), b), ~ \text{ for } [r] \in C_2.
\end{align*}
Let $0 \rightarrow M \rightarrow B' \rightarrow D \rightarrow 0$ denote the extension corresponding to the $2$-cocycle $(\alpha' , \beta')$. The equivalence between these two extensions are given by $\phi : B \rightarrow B', ~ (m, a) \mapsto (m + \gamma (a), a)$. Hence the map $H^2_{G, \mathrm{dend}}(D, M) \rightarrow \mathcal{E}xt(G, M)$ is well defined.

\medskip

For the converse part, let (\ref{exten}) be an extension of $D$ by $M$. Then $M$ is a submodule of $B$ via the inclusion map $i$. Let $s : D \rightarrow B$ be a section of the map $p$, i.e. $s$ satisfies $p \circ s= \mathrm{id}_{D}$. We define a map $\alpha \in \text{Maps}(G, \text{Hom}(\mathbb{K}[C_1]\otimes D,M))$ and a map $\beta \in \text{Hom}(\mathbb{K}[C_2]\otimes D^{\otimes 2},M)$ using the map $s$ as follows:
\begin{align}
\label{ext 1}  \alpha(g ; a) =~& s(a) - g ( s(g^{-1}a)),\\
\label{ext 2}  \beta([1]; a, b) = ~& \beta^l(a,b) = s(a \prec b) - s(a) \prec s(b),\\
\label{ext 3}  \beta([2]; a, b) =~& \beta^r(a,b) = s(a \succ b) - s(a) \succ s(b) .
\end{align}
Note that $p ( \alpha (g;a)) = p \circ s (a) - g ( p \circ s (g^{-1} a)) = 0$ (as $p \circ s= \mathrm{id}_{D}$). Hence $\alpha (g ; a) \in \mathrm{ker}(p) = \mathrm{ im }(i)$, therefore considered as an element in $M$.
We show that $(\alpha, \beta) \in Z^2_{G, \mathrm{dend}}(D,M)$. For this, we will verify the identities (\ref{coho 1}), (\ref{2-cocycle exp}) and (\ref{2-2co}). First, we have
\begin{align*}
g (\alpha(h ; {{g^{-1}}{a}}) ) + \alpha(g ; a) &= g (s(g^{-1}a))- gh (s (h^{-1}g^{-1}a)) + s(a) - g (s(g^{-1}a)) \\ 
                                               &= s(a)- gh (s (h^{-1}g^{-1}a)) = \alpha(gh; a).                                                               
\end{align*}
Hence the identity (\ref{coho 1}) follows. In the next, for $\varepsilon(g)= +1$, we get from Equation (\ref{ext 1}) that
\begin{align} \label{com 1}
s(a \prec b) &= \alpha (g; a \prec b) + g (s(g^{-1}a \prec g^{-1}b )) \\ \nonumber
&= \alpha (g ; a\prec b) + g ( s(g^{-1}a) \prec s(g^{-1}b) + \beta^l(g^{-1}a, g^{-1}b)) \\  \nonumber
& = \alpha (g; a\prec b) + g(s (g^{-1}a)) \prec g(s(g^{-1}b)) + g (\beta^l(g^{-1}a, g^{-1}b)).
\end{align}
On the other hand, from Equation (\ref{ext 2}), we get
\begin{align} \label{com 2}
s (a\prec b) &= s(a)\prec s(b) + \beta^l(a,b) \\ \nonumber
&= \big(g (s(g^{-1}a)) + i(\alpha(g; a))\big) \prec \big(g (s(g^{-1}b)) + i(\alpha(g; b))\big) + \beta^l(a, b) \\ \nonumber
&= g(s(g^{-1}a)) \prec g(s(g^{-1}b)) + a \prec \alpha(g;b) + \alpha(g;a) \prec b + \beta^l(a, b).
\end{align}
Comparing (\ref{com 1}) and (\ref{com 2}), we get
\begin{align*}
a \prec \alpha(g ; b) - \alpha(g ; a\prec b) + \alpha(g ;a) \prec b = g(\beta^l(g^{-1}a, g^{-1}b)) - \beta^l (a,b).
\end{align*}
Similarly, for $\varepsilon(g) = -1$, we have
\begin{align*}
a \prec \alpha(g; b) - \alpha(g; a\prec b) + \alpha(g;a) \prec b = g (\beta^r(g^{-1}b, g^{-1}a)) - \beta^l (a,b).
\end{align*}
One can repeat the computation as above for $[r] = [2]$. For $\varepsilon (g)= +1$ one have
\begin{align*}
a \succ \alpha(g; b) - \alpha(g; a\succ b) + \alpha(g;a) \succ b =g(\beta^r (g^{-1}a, g^{-1}b)) - \beta^r (a,b).
\end{align*}
and for $\varepsilon(g)= -1$ one have
\begin{align*}
a \succ \alpha(g; b) - \alpha(g ; a\succ b) + \alpha(g;a) \succ b = g(\beta^l(g^{-1}b, g^{-1}a)) - \beta^r (a,b).
\end{align*}
This proves the identity (\ref{2-cocycle exp}). Finally, $\beta$ is a $2$-cocycle in the ordinary dendriform algebra cohomology of $D$ with coefficients in $M$ follows from the standard argument \cite{das}.
Hence $(\alpha, \beta) \in Z^2_{G, \mathrm{dend}}(D,M)$. Like previous part, we can show show that equivalent extensions give rise to cohomologous $2$-cocycles. Hence the map $\mathcal{E}xt(D, M) \rightarrow H^2_{G, \mathrm{dend}}(D, M)$ is also defined. 

Finally, the above two maps are inverses to each other. Hence the proof.
\end{proof}

\section{Deformations of oriented dendriform algebras}\label{section-defor-ori-dend}
Gerstenhaber's deformation theory \cite{gers} for associative algebras has been extended to dendriform algebras in \cite{das}. In this section, we study an oriented version of the deformation theory for dendriform algebras.

Let $(D, \prec, \succ)$ be an oriented dendriform algebra over $(G, \varepsilon)$. Let $\phi : G \times D \rightarrow D$ denotes the action of $G$ on $D$, i.e. $\phi (g; a) = ga$, for $g \in G$ and $a \in D$. 
Consider the space $D[[t]]$ of formal power series in $t$ with coefficients from $D$.
\begin{defn} \label{deform defn}
A one-parameter formal deformation of $D$ consists of formal sums 
\begin{align*}
\prec_t = \sum_{i=0}^\infty t^i \prec_i ~~~ , ~~~ 
\succ_t = \sum_{i=0}^\infty t^i \succ_i  ~\in \mathrm{Hom} (D^{\otimes 2}, D) [[t]] ~~~ \text{ and } ~~~
\phi_t =  \sum_{i=0}^\infty t^i \phi_i  ~ \in \mathrm{Map}(G, \mathrm{Hom}(D,D))[[t]]
\end{align*}
in which $\prec_0 = \prec, ~\succ_0 =  \succ, ~ \phi_0 = \phi$ such that $(D[[t]], \prec_t, \succ_t)$ is a $\mathbb{K}[[t]]$-dendriform algebra and $\phi_t$ defines an orientation on this dendriform algebra over $(G, \varepsilon)$.

In other words, for each $n \geq 0$, $a,b,c \in D$ and $g \in G$, the following identities hold
\begin{itemize}
\item[(i)] $ \sum_{i+j = n} ( a \prec_i b ) \prec_j c  = \sum_{i+j = n} a \prec_i ( b \prec_j c + b \succ_j c),$\\
$ \sum_{i+j = n} ( a \succ_i b ) \prec_j c = \sum_{i+j = n}  a \succ_i ( b  \prec_j c ),$\\
$ \sum_{i+j = n} ( a \prec_i b + a \succ_i b) \succ_j c = \sum_{i+j = n} a \succ_i ( b \succ_j c), $
\item[(ii)] $\phi_n (gh; a) = \sum_{i+j = n} \phi_i ( g; \phi_j (h; a)),$ 
 \item[(iii)] $\sum_{i+j = n} \phi_i ( g ; \prec_j (a, b)) = \begin{cases}  \sum_{i+j+k = n} \prec_i ( \phi_j (g ; a), \phi_k (g ; b))  \quad  \text{ if } ~~ \varepsilon(g) = 1 \\
  \sum_{i+j+k = n} \succ_i ( \phi_j (g;b), \phi_k (g; a)) \quad  \text{ if } ~~ \varepsilon(g) = -1, \end{cases}$ 
  \item[(iv)] $\sum_{i+j = n} \phi_i ( g; \succ_j (a, b)) = \begin{cases}  \sum_{i+j+k = n} \succ_i ( \phi_j (g;a), \phi_k (g; b))  \quad \text{ if } ~~ \varepsilon(g) = 1 \\
  \sum_{i+j+k = n} \prec_i ( \phi_j (g;b), \phi_k (g; a)) \quad   \text{ if } ~~ \varepsilon(g) = -1. \end{cases}$
\end{itemize}
\end{defn}

\begin{defn} \label{equiv deform}
Two deformations $(\prec_t, \succ_t, \phi_t)$ and $(\prec_t', \succ_t', \phi_t')$ of an oriented dendriform algebra $D$ are said to be equivalent if there exists a formal isomorphism $\Psi_t = \sum_{i = 0}^\infty t^i \Psi_i : D[[t]] \rightarrow D[[t]]$ with $\Psi_0 = \mathrm{id}_D$ such that $\Psi_t$ defines an isomorphism between oriented dendriform algebras.
\end{defn}

Thus, the following identities must hold for $n \geq 0$, $a,b \in D$ and $g \in G$,
\begin{itemize}
\item[(i)] $\sum_{i+j = n } \Psi_i  ( a \prec_j' b ) = \sum_{i+j+k = n} \Psi_i (a) \prec_j \Psi_k (b)$,\\
$\sum_{i+j = n } \Psi_i  ( a \succ_j' b ) = \sum_{i+j+k = n} \Psi_i (a) \succ_j \Psi_k (b)$,
\item[(ii)] $\sum_{i+j = n} \Psi_i ( \phi_j' (g;a)) = \sum_{i+j = n} \phi_i (g; \Psi_j (a)). $
\end{itemize}

\begin{thm}\label{formal-def}
Let $(\prec_t, \succ_t, \phi_t)$ be a one-parameter formal deformation of an oriented dendriform algebra $D$. Then $( \xi_1, \pi_1)$ is a $2$-cocycle in the cohomology of the oriented dendriform algebra $D$ with coefficients in itself, where $\xi_1 \in \mathrm{Maps} (G, \mathrm{Hom}(D, D))$ is given by $\xi_1 (g; a) = \phi_1 (g; g^{-1} a)$, and 
$\pi_1 \in \mathrm{Hom} (\mathbb{K}[C_2] \otimes D^{\otimes 2}, D)$ is given by 
$\pi_1 ([1]; a, b) = a \prec_1 b$; $\pi_1 ([2]; a, b) = a \succ_1 b$.

In particular, if $\prec_i ~=~ \succ_i ~=~ \phi_i ~=~ 0$ for $0 < i < n$, then the pair $( \xi_n, \pi_n)$ is a $2$-cocycle, where $\xi_n$ and $\pi_n$ are defined in the similar way.

Moreover, the cohomology class of the corresponding $2$-cocycle depends only on the equivalence class of the deformation.
\end{thm}

\begin{proof}
We will verify (\ref{coho 1}), (\ref{2-cocycle exp}) and (\ref{2-2co}) to show that $(\xi_1, \pi_1)$ is a $2$-cocycle. It follows from condition (ii) of Definition \ref{deform defn} that
\begin{align}\label{last-proof-1}
\phi_1 ( gh; a) = g ( \phi_1 (h; a)) + \phi_1 (g; ha).
\end{align}
From the definition of $\xi_1$, we get
\begin{align*}
g ( \xi_1 ( h; g^{-1} a)) + \xi_1 (g; a) =~&  g (\phi_1 (h; h^{-1} g^{-1} a)) + \phi_1 ( g ; g^{-1} a) \\
\stackrel{(\ref{last-proof-1})}{=}~&   \phi_1 ( gh ; h^{-1} g^{-1} a) = \xi_1 (gh; a).
\end{align*}
Hence the identity (\ref{coho 1}) follows. For $n=1$, from conditions (iii) and (iv) of Definition \ref{deform defn} we get
\begin{align} \label{pre compare 1}
  g(a \prec_1 b) + \phi_1 ( g; a \prec b) = \begin{cases}   ga \prec \phi_1 (g, b) +  \phi_1 (g;a) \prec gb 
 +   ga \prec_1  g b \quad  \text{ if } ~~ \varepsilon(g) = +1 \\
    gb \succ \phi_1 (g; a) +  \phi_1 (g ; b) \succ g a +  gb \succ_1 g a \quad  \text{ if } ~~ \varepsilon(g) = -1, \end{cases} 
\end{align}
\begin{align} \label{pre compare 2}
       g(a \succ_1 b) + \phi_1 ( g; a \succ b)& = \begin{cases} ga \succ \phi_1 (g; b) +  \phi_1 (g;a) \succ gb 
  +   ga \succ_1  g b \quad  \text{ if } ~~ \varepsilon(g) = +1 \\
  gb \prec \phi_1 (g; a) +  \phi_1 (g ; b) \prec g a +  gb \prec_1 g a \quad  \text{ if } ~~ \varepsilon(g) = -1. \end{cases}
\end{align}
From the first part of (\ref{pre compare 1}) and from the second part of (\ref{pre compare 2}), we obtain
\begin{align*}
 a \prec \xi_1 (g; b) - \xi_1 ( g; a \prec b ) + \xi_1 (g;a) \prec b = \begin{cases} g ( g^{-1} a \prec_1 g^{-1} b ) - a \prec_1 b  ~& \mathrm{ if }~~ \varepsilon (g) = +1 \\
g ( g^{-1} b \succ_1 g^{-1} a) -  a \prec_1 b  ~& \mathrm{ if }~~ \varepsilon (g) = -1. \end{cases}
\end{align*}
Similarly, from the first part of (\ref{pre compare 2}) and from the second part of (\ref{pre compare 1}), we get
\begin{align*}
 a \succ \xi_1 (g; b) - \xi_1 ( g; a \succ b ) + \xi_1 (g;a) \succ b = \begin{cases} g ( g^{-1} a \succ_1 g^{-1} b ) - a \succ_1 b  ~& \mathrm{ if }~~ \varepsilon (g) = +1 \\
 g ( g^{-1} b \prec_1 g^{-1} a) - a \succ_1 b  ~& \mathrm{ if }~~ \varepsilon (g) = -1. \end{cases}
\end{align*}
Hence the identity (\ref{2-cocycle exp}) follows. Finally, from the ordinary formal deformations of dendriform algebras, it follows that $\pi_1 \in \mathrm{Hom}(\mathbb{K}[C_2] \otimes D^{\otimes 2}, D)$ is a $2$-cocycle in the ordinary dendriform algebra cohomology of $D$ with coefficients in itself, i.e. (\ref{2-2co}) holds. Therefore, 
$( \xi_1, \pi_1) \in Z^2_{G, \mathrm{dend}} (D, D)$ defines a $2$-cocycle.  In a similar way, if $\prec_i = \succ_i = \phi_i = 0$, for $0 < i < n$, then one can verify that $( \xi_n, \pi_n)$ is a $2$-cocycle.

Next, suppose $(\prec_t, \succ_t, \phi_t)$ and $(\prec_t', \succ_t', \phi_t')$ are two equaivalent deformations of an oriented dendriform algebra $D$. For $n=1$, from conditions (i) and (ii) of Definition \ref{equiv deform} we have
\begin{align*}
  a \prec_1' b  + \Psi_1  ( a \prec b )=~&  a \prec_1 b + \Psi_1 (a) \prec b + a \prec \Psi_1 (b),\\
 a \succ_1' b  + \Psi_1  ( a \succ b )=~&  a \succ_1  b + \Psi_1 (a) \succ  b + a \succ \Psi_1 (b), \\
  \phi_1' (g;a) + \Psi_1 ( ga) =~&  g \Psi_1 (a) + \phi_1 (g; a). 
\end{align*}
The first two conditions is equivalent to
\begin{align*}
(\pi_1' - \pi_1)([r]; a, b) = \theta_1 ([r]; a, \Psi_1 (b)) - \Psi_1 ( \pi_D ([r]; a, b)) + \theta_2 ([r]; \Psi_1 (a), b), ~ \text{ for } [r] \in C_2
\end{align*}
and the third condition is equivalent to $(\xi_1' - \xi_1) (g;a) = g (\Psi_1 (g^{-1} a)) - \Psi_1 (a).$ Therefore, by (\ref{when-cob}), the difference $(\xi_1', \pi_1') - (\xi_1, \pi_1)$ is a $2$-coboundary. Hence, the corresponding cohomology classes are the same.
\end{proof}

To obtain a one-to-one correspondence between equivalence classes of certain deformations and the cohomology group $H^2_{G, \mathrm{dend}} (D, D)$, we define a truncated version of formal deformations.

\begin{defn}
An infinitesimal deformation of an oriented dendriform algebra $D$ is a deformation of $D$ over the base $\mathbb{K}[[t]]/ (t^2).$
\end{defn}

Thus, an infinitesimal deformation of an oriented dendriform algebra $D$ consists of sums 
$\prec_t = \prec + t \prec_1,$ $\succ_t = \succ + t \succ_1$ and $\phi_t = \phi + t \phi_1$ such that $(D[[t]], \prec_t,\succ_t, \phi_t)$ is an oriented dendriform algebra over the base $\mathbb{K}[[t]]/ (t^2)$.  Then we have the following theorem.

\begin{thm}\label{inf-def}
There is a one-to-one correspondence between the space of equivalence classes of infinitesimal deformations and the second cohomology group $H^2_{G, \mathrm{dend}} (D, D)$.
\end{thm}

\section{Further discussions} 
In this section, we collect some relevant questions that need further study.

(A) In \cite{fard-guo} Ebrahimi-Fard and Guo construct a functor from the category of dendriform algebras to the category of Rota-Baxter algebras. This is adjoint to the functor that gives rise to the dendriform structure from a Rota-Baxter operator. It could be interesting to verify whether these two functors restrict to functors between the category of involutive (resp. oriented) Rota-Baxter algebras and the category of involutive (resp. oriented) dendriform algebras.

(B) In Subsection \ref{subsec-inv-dend-inf}, we have discussed involutive $\mathrm{Dend}_\infty$-algebras that are closely related to the cohomology of involutive dendriform algebras (Remark \ref{remark-inf-cohomo}). It could be nice to find oriented $\mathrm{Dend}_\infty$-algebras (also $A_\infty$-algebras) and their relationship with the cohomology of oriented dendriform algebras as described in Section \ref{section-cohomo-ori-dend}.

(C) The cohomology associated to a Rota-Baxter operator on an associative has been defined in \cite{das-rota}. It has been shown that the relation between Rota-Baxter operators and dendriform algebras passes onto the level of cohomology. In a forthcoming paper, we aim to explore Rota-Baxter operators on involutive associative algebras and their relations with the cohomology of involutive dendriform algebras introduced in the present paper.

(D) Recently, Rota-Baxter family of algebras and dendriform family of algebras are introduced with motivation from the renormalization in the quantum field theory \cite{zhang-gao}. In a subsequent paper, we aim to construct the cohomology theory for a dendriform family of algebras and their deformation theory. It is not hard to define involutions (resp. orientations) in a dendriform family of algebras. We also aim to demonstrate the results of the present paper in the context of an oriented dendriform family of algebras.

\medskip


\noindent {\em Acknowledgements.} The research of A. Das is supported by the fellowship of Indian Institute of Technology (IIT) Kanpur. The author thanks the Institute for financial support.


\begin{thebibliography}{BFGM03}

\bibitem{aguiar} M. Aguiar, Pre-Poisson algebras, {\em Lett. Math. Phys.} 54 (2000), no. 4, 263-277.

\bibitem{aguiar2} M. Aguiar, Infinitesimal bialgebras, pre-Lie and dendriform algebras, In: {\em Hopf algebras,} 1-33, Lecture Notes in Pure and Appl. Math., 237, {\em Dekker, New York}, 2004.

\bibitem{aguiar-loday}
M. Aguiar and J.-L. Loday, Quadri-algebras,
{\em J. Pure Appl. Algebra} 191 (2004), no. 3, 205-221.

\bibitem{braun} C. Braun, Involutive $A_\infty$-algebras and dihedral cohomology, {\em J. Homotopy Relat. Struct.} 9 (2014), no. 2, 317-337.

\bibitem{cos} K. Costello, Topological conformal field theories and Calabi-Yau categories, {\em Adv. Math.} 210(1) (2007), 165-214.

\bibitem{das-rota} A. Das, Deformations of associative Rota-Baxter operators, {\em J. Algebra} 560 (2020) 144-180.

\bibitem{das-loday} A. Das, Deformations of Loday-type algebras and their morphisms, {\em J. Pure Appl. Algebra} to appear, {\em arXiv:1904:12366}

\bibitem{das} A. Das, Cohomology and deformations of dendriform algebras, and $\mathrm{Dend}_\infty$-algebras, {\em preprint, arXiv:1903.11802} 

\bibitem{das-saha} A. Das and R. Saha, On equivariant dendriform algebras,  {\em Colloq. Math.} to appear, {\em arXiv:1912.09221}

\bibitem{fa} K. Ebrahimi-Fard, Loday-Type algebras and the Rota–Baxter relation, {\em Lett. Math. Phys.} 61 (2002), 139-147.

\bibitem{fard-guo} K. Ebrahimi-Fard and L. Guo, Rota-Baxter algebras and dendriform algebras, {\em J. Pure Appl. Algebra} 212 (2008), 320-339.

\bibitem{fer-gian} R. Fern\'{a}ndez-Val\'{e}ncia and J. Giansiracusa, On the Hochschild homology of involutive algebras, {\em Glasg. Math. J.} 60 (2018), no. 1, 187-198. 

\bibitem{gers} M. Gerstenhaber, On the deformation of rings and algebras, {\em Ann. of Math.} (2) 79 (1964), 59-103.

\bibitem{koam-pira} Ali. N. A. Koam and T. Pirashvili, Cohomology of oriented algebras, {\em Comm. Algebra} 46 (2018), no. 7, 2947-2963.


\bibitem{leroux}
P. Leroux, Ennea-algebras, {\em J. Algebra} 281 (2004), no. 1, 287-302.


\bibitem{loday} J.-L. Loday, Dialgebras, {\em Dialgebras and related operads,} 7-66,
Lecture Notes in Math., 1763, {\em Springer, Berlin,} 2001.

\bibitem{loday-ronco}
J.-L. Loday and M. Ronco, Trialgebras and families of polytopes, {\em Homotopy theory: relations with algebraic geometry, group cohomology, and algebraic K-theory}, 369–398, 
Contemp. Math., 346, {\em Amer. Math. Soc., Providence, RI,} 2004. 

\bibitem{loday-val} J.-L. Loday and B. Vallette, Algebraic operads, {\em Springer, Heidelberg,} 2012. +xxiv+634 pp. ISBN: 978-3-642-30361-6


\bibitem{nov-thibon} J.-C. Novelli and J.-Y. Thibon, Polynomial realizations of some trialgebras, in: Formal Power Series and Algebraic Combinatorics (FPSAC), San Diego, California, 2006.

\bibitem{ronco} M. Ronco, Shuffle bialgebras, {\em Ann. Inst. Fourier (Grenoble)} 61 (2011), no. 3, 799-850.

\bibitem{weibel} C. Weibel, An introduction to homological algebra,
Cambridge Studies in Advanced Mathematics, 38. {\em Cambridge University Press, Cambridge,} 1994.

\bibitem{zhang-gao} 
Y. Zhang and X. Gao, Free Rota-Baxter family algebras and (tri)dendriform family algebras, {\em Pacific J. Math.} , 301 (2019), No. 2, 741-766.

\end{thebibliography}
\end{document}